\documentclass[11pt,a4paper]{article}

\usepackage[T1]{fontenc}
\usepackage{amsmath,amsthm,amsopn,amsfonts,amssymb,dsfont,color}
\usepackage{mathrsfs}
\usepackage{statex}
\usepackage{graphicx}
\usepackage{float}
\usepackage{xcolor}
\usepackage{mathrsfs}
\usepackage{mathtools}
\usepackage{accents}
\usepackage{setspace}
\usepackage{times}
\usepackage{listings}
\usepackage{enumerate}
\usepackage{indentfirst}
\usepackage{a4,a4wide}
\usepackage{tikz}

\newtheorem{theorem}{\textbf{Theorem}}[section]
\newtheorem{lemma}[theorem]{\textbf{Lemma}}
\newtheorem{proposition}[theorem]{\textbf{Proposition}}
\newtheorem{claim}[theorem]{\textbf{Claim}}

\newtheorem{remark}[theorem]{\textbf{Remark}}

\numberwithin{equation}{section}
\numberwithin{figure}{section}

\makeatletter
\g@addto@macro\th@plain{\thm@headpunct{}}
\makeatother

\newcommand\dps{\displaystyle}

\newcommand\bea{\begin{eqnarray}}
\newcommand\eea{\end{eqnarray}}
\newcommand\beaa{\begin{eqnarray*}}
\newcommand\eeaa{\end{eqnarray*}}

\title{On the propagation speed
 of the single monostable equation}
\date{}
\author{}
\begin{document}

\maketitle
\vspace{-30pt}
\begin{center}
{\large\bf
Chang-Hong Wu
\footnote{Department of Applied Mathematics, National Yang Ming Chiao Tung University, Hsinchu, Taiwan.

e-mail: {\tt changhong@math.nctu.edu.tw}}
\footnote{National Center for Theoretical Sciences, Taipei, Taiwan},
Dongyuan Xiao\footnote{Graduate School of Mathematical Science, The University of Tokyo, Tokyo, Japan.

e-mail: {\tt dongyuanx@hotmail.com}} and
Maolin Zhou\footnote{Chern Institute of Mathematics and LPMC, Nankai University, Tianjin, China.

e-mail: {\tt zhouml123@nankai.edu.cn}}
} \\
[2ex]
\end{center}

\begin{abstract}
In this paper, we first focus on the speed selection problem for the reaction-diffusion equation of the monostable type. By investigating the decay rates of the minimal traveling wave front, we propose a sufficient and necessary condition that reveals the essence of propagation phenomena. Moreover, since our argument relies solely on the comparison principle, it can be extended to more general monostable dynamical systems, such as nonlocal diffusion equations.
\\


\end{abstract}

\section{Introduction}

In 1937, Fisher \cite{Fisher} and Kolmogorov et al. \cite{KPP} introduced the Fisher-KPP equation:
\beaa
w_t=w_{xx}+f(w)=w_{xx}+w(1-w),\ \ t>0,\ x\in\mathbb{R},
\eeaa
to describe the spatial propagation of organisms, such as dominant genes and invasive species, in a homogeneous environment. It is well-known, as demonstrated in \cite{Fisher, KPP}, that under the KPP condition:
\begin{equation}\label{kpp condition}
f'(0)w\geq f(w)\ \text{for all}\ w\in[0,1],
\end{equation}
the spreading speed of \eqref{scalar equation} can be directly derived from its linearization at $w=0$:
$$w_t=w_{xx}+w.$$
This so-called "linear conjecture," which posits that nonlinear differential equations governing population spread always have the same velocity as their linear approximation, has been developed over more than 80 years through numerous examples. It is explicitly stated by Bosch et al. \cite{Bosch Metz Diekmann} and Mollison \cite{Mollison}.

For the general monostable equation,
\begin{equation}\label{scalar equation}
\left\{
\begin{aligned}
&w_t=w_{xx}+f(w),\ \ t>0,\ x\in\mathbb{R},\\
&w(0,x)=w_0(x),\ \ x\in\mathbb{R},
\end{aligned}
\right.
\end{equation}
where 
$f$ satisfying
\begin{equation*}
f(0)=f(1)=0,\ f'(0)>0>f'(1),\ \text{and}\ f(w)>0\ \text{for all}\ w\in(0,1),
\end{equation*}
Aronson and Weinberger \cite{Aronson Weinberger} showed the existence of a speed $$c^*\geq 2\sqrt{f'(0)}>0$$ indicating the 
spreading property
of the solution to the Cauchy problem \eqref{scalar equation}
as follows:
\begin{equation}\label{def:spreading speed}
\left\{
\begin{aligned}
&\lim_{t\to\infty}\sup_{|x|\ge ct}w(t,x)=0\ \ \mbox{for all} \ \ c>c^*;\\
&\lim_{t\to\infty}\sup_{|x|\le ct}|1-w(t,x)|=0\ \ \mbox{for all} \ \ c<c^*.
\end{aligned}
\right.
\end{equation}
\begin{itemize}
\item For the case $c^*=2\sqrt{f'(0)}$, it is called spreading speed linear selection; 
\item For the case $c^*>2\sqrt{f'(0)}$, it is called spreading speed nonlinear selection. 
\end{itemize}
We remark that, in general, the minimal speed $c^*$ depends on the shape of $f$ and cannot be characterized explicitly.

A typical example
for understanding the link between speed linear selection and nonlinear selection of a single reaction-diffusion equation
is as follows 
\cite{Hadeler Rothe}:
\bea\label{Hadeler-Rothe eq}
\partial_t w-w_{xx}=w(1-w)(1+sw),
\eea
where $s\geq 0$ is the continuously varying parameter.  Moreover, the KPP condition \eqref{kpp condition} is satisfied if and only if $0\le s\le 1$.
If $s>1$, \eqref{kpp condition} is not satisfied, and such 
$f(w;s)$
is called the weak Allee effect.
The minimal traveling wave speed $c^*(s)$ is characterized in \cite{Hadeler Rothe} as:
\begin{equation*}
c^*(s)=
\begin{cases}
\dps 2&\quad\text{if}\quad 0\le s\le 2,\\
\dps \sqrt{\frac{2}{s}}+\sqrt{\frac{s}{2}}&\quad\text{if}\quad s>2.
\end{cases}
\end{equation*}
Then it is easy to see that the minimal speed  $c^*(s)$ is linearly selected for $0<s\le2$, while it is nonlinearly selected for $s>2$. Note particularly that, for $s\in(1,2]$, the minimal speed $c^*(s)$ is still linearly selected
even though the KPP condition \eqref{kpp condition} is not satisfied. 
In addition, we see that the transition front from linear selection to nonlinear selection for \eqref{Hadeler-Rothe eq} occurs when $s=2$.

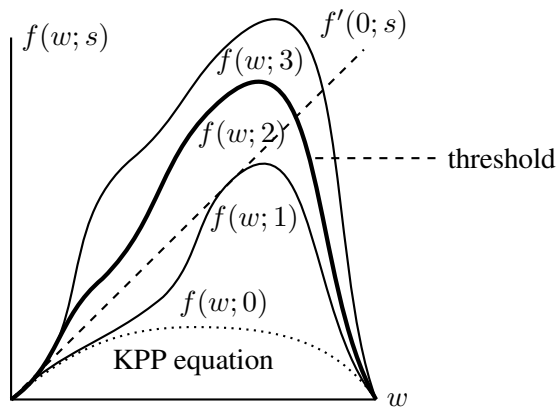
\begin{figure}
\begin{center}
\begin{tikzpicture}[scale = 0.8]
\draw[thick](-3,6)-- (-3,0)-- (3,0) node[right] {$w$};
\draw[dashed] [thick](-3,0)--(2.8,5.8) node[above] {$f'(0;s)$};
\node[right] at (-3,6) {$f(w;s)$};
\draw[dotted] [thick] (-3,0) to [out=40,in=180] (0,1.2) to [out=0,in=125] (3,0);
\node[below] at (0.5,2) {$f(w;0)$};
\node[below] at (0,1) {KPP equation};
\draw [thick] (-3,0) to [out=45, in=220] (-0.5,1.7) to [out=45,in=220] (0.7,3.7) to [out=40,in=130] (3,0);
\node[below] at (1,3.4) {$f(w;1)$};
\draw [ultra thick] (-3,0) to [out=35, in=220] (-1.5,2) to [out=45,in=220] (0.5,5) to [out=40,in=120] (3,0);
\node[below] at (0.8,4.8) {$f(w;2)$};
\draw [thick] (-3,0) to [out=35, in=220] (-1,4) to [out=40,in=220] (0.7,6) to [out=40,in=110] (3,0);
\node[below] at (1.1,6) {$f(w;3)$};
\draw[dashed] [thick] (2,4)--(4,4) node[right] {threshold};
\end{tikzpicture}
\caption{The transition from linear selection to nonlinear selection of \eqref{Hadeler-Rothe eq}.}
\end{center}
\end{figure}




In the remarkable work \cite{Lucia Muratov Novaga},
Lucia, Muratov, and Novaga proposed a variational approach to rigorously establish a mechanism to determine speed linear selection and nonlinear selection on the single monostable reaction-diffusion equations. Roughly speaking, the following two conditions are equivalent:
\begin{itemize}
\item[(i)] the minimal traveling wave speed of
$w_t=w_{xx}+f(w)$
is nonlinearly selected;
\item[(ii)] $\Phi_c[w]\leq 0$ holds for some $c>2\sqrt{f'(0)}$ and $w(\not\equiv 0)\in C_0^\infty(\mathbb{R})$, where
\beaa
\Phi_c[w]:=\int_\mathbb{R}e^{cx}\Big[\frac{1}{2}w_x^2-\int_0^wf(s)ds\Big]dx.
\eeaa
\end{itemize}
As an application, some explicit and easy-to-check results can be obtained to determine linear and nonlinear selection (see Section 5 in \cite{Lucia Muratov Novaga}). A related issue can be found in \cite{Ma Ou} using the theory of abstract monotone semiflow.

The first part of this paper is dedicated to the speed selection problem of the single monostable reaction-diffusion equations. 
We establish a new mechanism for determining the linear or nonlinear selection by considering a family of continuously varying nonlinearities.
By varying the parameter within the nonlinearity, we obtain a full understanding of how the decay rate of the minimal traveling wave
at infinity influences the propagation speed.
Unlike the mechanism established by Lucia et al., our approach provides insight into the process by which linear selection evolves into nonlinear selection.
The propagation phenomenon and inside dynamics
of the front for more general single equations have been widely discussed in the literature.
We may refer to, e.g., \cite{Berestycki Hamel2012, Fife McLeod1977, Gardner et al 2012, Liang Zhao 2007, Roques et al 2012, Rothe1981, Stokes1976} and references cited therein.

Furthermore,
as noted in \cite{van Saarloos2003}, many natural elements such as
advection, nonlocal diffusion, and periodicity need to be considered in the propagation problem.  The variational approach, 
as discussed in \cite{Lucia Muratov Novaga}, can 
treat homogeneous single equations with the standard Laplace diffusion, 
but it is not easy to handle the equation without a variational approach.
In contrast, our method can be applied to the equations and systems as long as the comparison principle holds.
In the second part of this paper, we extend our observation on the threshold behavior between linear selection and nonlinear selection to the single nonlocal diffusion equation.


\subsection{Main results on the reaction-diffusion equation}

\noindent

We first consider the following single equation
\beaa
w_t= w_{xx} +f(w;s),
\eeaa
where
$\{f(\cdot;s)\}\subset C^2$ is a one-parameter family of nonlinear functions satisfying monostable condition and varies continuously and monotonously on the parameter $s\in[0,\infty)$. 
The assumptions on $f$ are as follows:
\begin{itemize}
    \item[(A1)](monostable condition) $f(\cdot;s)\in C^2([0,1])$, $f(0;s)=f(1;s)=0$, $f'(0;s):=\gamma_0>0>f'(1;s)$,  and $f(w;s)>0$ for all $s\in\mathbb{R}^+$ and $w\in(0,1)$.
    \item[(A2)] (Lipschitz continuity) $f(\cdot;s)$, $f'(\cdot;s)$, and $f''(\cdot;s)$ are Lipschitz continuous on $s\in\mathbb{R}^+$
    uniformly in $w$. In other words,
    there exists $L_0>0$ such that
    \beaa
    |f^{(n)}(w;s_1)-f^{(n)}(w;s_2)|\leq L_0|s_1-s_2|\quad \mbox{for all $\ $ $w\in[0,1]$ and $n=0,1,2$,}
    \eeaa
    where $f^{(n)}$ mean the $n$th derivative of $f$ with respect to $w$ for $n\in\mathbb{N}$, {\it i.e.}, $f^{(0)}=f$, $f^{(1)}=f'$, and $f^{(2)}=f''$.
    \item[(A3)] (monotonicity condition) $f(w;\hat{s})>f(w;s)$ for all $w\in (0,1)$ if $\hat{s}>s$, and $f''(0;\hat{s})> f''(0;s)$ if $\hat{s}>s$.
\end{itemize}

\begin{remark}
Without loss of generality, we assume $\gamma_0=1$ in the assumption (A1) for the part concerned with the single reaction-diffusion equation, such that the linearly selected spreading speed is equal to $2$. 
\end{remark}

\begin{remark}
Note that, in this paper, we always assume $\{f(\cdot;s)\}\subset C^2$ as that in the assumption (A1) for the simplicity of the proof. As a matter of fact,
our approach still works for weaker regularity of $f$, say
$\{f(\cdot;s)\}\subset C^{1,\alpha}$ for some $\alpha\in(0,1)$.
If we consider a higher degree of regularity for $f$, such as ${f(\cdot;s)}\subset C^k$ for some $k>2$, then the condition in the assumption  (A3) for $f''(0;\cdot)$ will be replaced by $f^{(i)}(0;\cdot)$ for some $1<i\le k$.
\end{remark}


Thanks to the assumption  (A1), there exists the minimal traveling wave speed for all $s\in[0,\infty)$, denoted by $c^*(s)$, such that
the system
\begin{equation}\label{scalar tw-parameter s}
\left\{
\begin{aligned}
&W''+cW'+f(W;s)=0,\quad  \xi\in\mathbb{R},\\
&W(-\infty)=1,\ W(+\infty)=0,\\
&W'<0,\quad \xi\in\mathbb{R},
\end{aligned}
\right.
\end{equation}
admits a unique (up to translations) solution $(c,W)$ if and only if  $c\geq c^*(s)$, where $c^*(s)$ is the spreading speed defined as \eqref{def:spreading speed}. 
In the literature, the minimal traveling wave $(c^*,W)$ is classified into two types: {\em pulled front} and
{\em pushed front} \cite{Rothe1981,Stokes1976,van Saarloos2003}.
\begin{itemize}
    \item The minimal traveling wave $W(\xi)$ with the speed $c^*$ is called a {\em pulled front} if $c^*=2\sqrt{f'(0)}$.
In this case, the front is pulled by the leading edge with speed determined by its linearization at the unstable state $w=0$. Therefore, the minimal traveling wave speed $c^*$ is also said to be linearly selected. 
\item On the other hand, if $c^*>2\sqrt{f'(0)}$,
the minimal traveling wave $W(\xi)$ with a speed $c^*$ is called a {\em pushed front} since
the propagation speed is determined by the whole wave, not only by the behavior of the leading edge. Thus
the minimal traveling wave speed $c^*$ is also said to be nonlinearly selected.
\end{itemize}

We further assume that linear (resp., nonlinear) selection mechanism can occur at some $s$.
More precisely, $f(\cdot;s)$ satisfies 
\begin{itemize}
    \item[(A4)] there exists $s_1>0$ such that $f(w;s_1)$ satisfies KPP condition \eqref{kpp condition},
   and thus $c^*(s_1)=2$.
    \item[(A5)]  there exists  
     $s_2>s_1$ such that $c^*(s_2)>2$.
\end{itemize}

\begin{remark}\label{rk:a4+a5}
In view of the assumption (A3), a simple comparison yields that $c^*(\hat{s})\geq c^*(s)$ if $\hat{s}\geq s$.
Together with assumptions (A4), (A5) and the fact $c^*(s)\geq 2$ for all $s\geq0$, we see that:
\begin{itemize}
\item[(1)] $c^*(s)=2$ for all $0\leq s\leq s_1$;
\item[(2)] $c^*(s)>2$ for all $s\geq s_2$.
\end{itemize}
\end{remark}

\begin{remark}
It is easy to check that \eqref{Hadeler-Rothe eq} satisfies assumptions  (A1)-(A5).
\end{remark}

Our first main result describes how the speed linear selection
evolves to the speed nonlinear selection in terms of the varying parameter $s$.
The key point is to completely characterize the evolution of the decay rate of the minimal traveling wave $W_s(\xi)$ with respect to $s$.
It is well known (\cite{Aronson Weinberger}) that if $c^*(s)=2$, then
\bea\label{decay-W-linear}
W_s(\xi)=A\xi e^{-\xi}+Be^{-\xi}+o(e^{-\xi})\quad \mbox{as $\xi\to+\infty$},
\eea
where 
$A\geq0$ and $B\in \mathbb{R}$,
and $B>0$ if $A=0$. 
As we will see, the key point to understanding the speed selection problem is to determine the leading order of the decay rate of $W_s(\xi)$, {\it i.e.},
whether $A>0$ or $A=0$ in \eqref{decay-W-linear}.

\begin{theorem}\label{th: threshold scalar equation}
Assume that assumptions (A1)-(A5) hold. Then there exists the threshold value $s^*\in[s_1,s_2)$ such that the minimal traveling wave speed of \eqref{scalar tw-parameter s} satisfies
\begin{equation}\label{def of threshold scalar}
c^*(s)=2\quad\text{for all}\ s\in[0,s^*];\quad c^*(s)>2\quad\text{for all}\ s\in(s^*,\infty).
\end{equation}
Moreover,
the minimal traveling wave $W_s(\xi)$ satisfies 
\begin{equation}\label{asy tw threshold scalar}
W_s(\xi)=Be^{-\xi}+o(e^{-\xi})\quad \mbox{as}\quad \xi\to+\infty\quad\text{for some}\quad B>0,
\end{equation}
if and only if $s=s^*$.
\end{theorem}

\begin{remark}
\begin{itemize}
\item[(1)]
Note that \eqref{asy tw threshold scalar} in Theorem~\ref{th: threshold scalar equation} indicates that, as $\xi\to+\infty$, the leading order of the decay rate of $W_s(\xi)$ switches from $\xi e^{-\xi}$ to $e^{-\xi}$ as $s\to s^*$ from below.
\item[(2)] In our proof of \eqref{def of threshold scalar} and the sufficient condition for \eqref{asy tw threshold scalar},
the condition in the assumption  (A3) that $f''(0;\hat{s})> f''(0;s)$ for $\hat{s}>s$ is not required.
\end{itemize}
\end{remark}

\medskip
The classification of traveling wave fronts for \eqref{scalar tw-parameter s} is well-known. We summarize the results as follows:
\begin{proposition}\label{prop: classification scalar}
Assume $f(\cdot)$ satisfies the monostable condition. The traveling wave fronts $(c,W)$, defined as \eqref{scalar tw-parameter s}, satisfies
\begin{itemize}
\item[(1)] there exists $(A,B)\in \mathbb{R}^+\times\mathbb{R}$ or $A=0,B>0$  such that $W(\xi)=A\xi e^{-\xi}+B e^{-\xi}+o(e^{-\xi})$ as $\xi\to+\infty$, 
if and only if $c=c^*=2$;
\item[(2)] there exists $A>0$ such that $W(\xi)=Ae^{-\lambda^+(c)\xi}+o(e^{-\lambda^+(c)\xi})$ as $\xi\to+\infty$, 
if and only if $c=c^*>2$;
\item[(3)] there exists $A>0$ such that $W(\xi)=Ae^{-\lambda^-(c)\xi}+o(e^{-\lambda^-(c)\xi})$ as $\xi\to+\infty$, 
if and only if $c>c^*$.
\end{itemize}
Here, $\lambda^{\pm}(c)$ are defined as
\bea\label{lambda + - scalar}
\lambda^{\pm}(c):=\frac{c\pm\sqrt{c^2-4}}{2}>0.
\eea
\end{proposition}

\begin{remark}
Combining \eqref{decay-W-linear}, Theorem \ref{th: threshold scalar equation}, and Proposition \ref{prop: classification scalar},
we can fully understand how the decay rates of the minimal traveling wave depend on $s$, which is formulated as follows:
\begin{itemize}
    \item[(1)] Pulled front: if $s\in[0,s^*)$, then $W_s(\xi)=A\xi e^{-\xi}+Be^{-\xi}+o(e^{-\xi})$ as $\xi\to+\infty$ with $A>0$ and $B\in \mathbb{R}$;
    \item[(2)] Pulled-to-pushed transition front: if $s=s^*$, then $W_s(\xi)=Be^{-\xi}+o(e^{-\xi})$ as $\xi\to+\infty$ with $B>0$;
    \item[(3)] Pushed front: if $s\in(s^*,\infty)$, then $W_s(\xi)=Ae^{-\lambda_s^+\xi}+o(e^{-\lambda_s^+\xi})$ as $\xi\to+\infty$ with $A>0$.
\end{itemize}
Here $\lambda_{s}^+$ is defined as \eqref{lambda + - scalar} with speed $c=c^*(s)$.

\end{remark}

\subsection{Main results on the nonlocal equation}\label{sec: intro nonlocal}

\noindent

Next, we  consider the following single nonlocal diffusion equation
\beaa
w_t= J\ast w-w +f(w;q),
\eeaa
where
$\{f(\cdot;q)\}\subset C^2$ is a one-parameter family of nonlinear functions satisfying assumptions (A1)-(A3) defined in \S 1.1 with $s=q$, 
$J$ is a nonnegative dispersal kernel defined on $\mathbb{R}$, and $J\ast w$ is defined as
$$J\ast w(x):=\int_{\mathbb{R}}J(x-y)w(y)dy.$$ 
For the simplicity of our discussion, throughout this paper, we always assume that the dispersal kernel
\begin{equation}\label{assumption on J}
J\ \text{is compactly supported, symmetric, and}\ \int_{\mathbb{R}}J=1.
\end{equation}


Under the assumption (A1), it has been proved in \cite{Coville} that there exists the minimal traveling wave speed for all $q\in[0,\infty)$, denoted by $c_{NL}^*(q)$, such that
the system
\begin{equation}\label{scalar nonlocal tw-parameter s}
\left\{
\begin{aligned}
&J\ast \mathcal{W}+c\mathcal{W}'+f(\mathcal{W};q)-\mathcal{W}=0,\quad  \xi\in\mathbb{R},\\
&\mathcal{W}(-\infty)=1,\ \mathcal{W}(+\infty)=0,\\
&\mathcal{W}'<0,\quad \xi\in\mathbb{R},
\end{aligned}
\right.
\end{equation}
admits a unique (up to translations) solution $(c,\mathcal{W})$ if and only if
$c\geq c_{NL}^*(q)$. 
Moreover, there is a lower bound estimate for the minimal speed $c_{NL}^*(q)\geq c_0^*$,
where the critical speed $c_0^*$ is given by the following variational formula
\begin{equation}\label{formula of c_NL}
c_{0}^*:=\min_{\lambda>0}\frac{1}{\lambda}\Big(\int_{\mathbb{R}}J(x)e^{\lambda x}dx+f'(0;q)-1\Big),
\end{equation}
which derived from the linearization of \eqref{scalar nonlocal tw-parameter s} at the trivial state $\mathcal{W}= 0$.
If $f(\cdot;q)$ also satisfies the KPP condition \eqref{kpp condition}, then
$c_{NL}^*(q)= c_0^*$. Therefore, we call the case $c_{NL}^*(q)= c_0^*$ as the speed linear selection and the case $c_{NL}^*(q)> c_0^*$ as the speed nonlinear selection.

\begin{remark}\label{rm:lambda_0}
Let $h(\lambda)$ be defined by
$$h(\lambda):=\int_{\mathbb{R}}J(z)e^{\lambda z}dz-1+f'(0;q).$$
It is easy to check that $\lambda\to h(\lambda)$ is an increasing, strictly convex, and sublinear function satisfying $h(0)=f'(0;q)>0$.
Therefore, there exist only one $\lambda_0>0$ satisfying $h(\lambda_0)=c_0^*\lambda_0$, and for $c>c_0^*$, the equation $h(\lambda)=c\lambda$ admits two different positive roots $\lambda_q^-(c)$ and $\lambda_q^+(c)$ satisfying $0<\lambda_q^-(c)<\lambda_0<\lambda_q^+(c)$. 
\end{remark}

We further assume that a linear (resp., nonlinear) selection mechanism can occur at some $q$.
More precisely, $f(\cdot;q)$ satisfies 
\begin{itemize}
    \item[(A6)] there exists $q_1>0$ such that $f(w;q_1)$ satisfies KPP condition \eqref{kpp condition},
   and thus $c_{NL}^*(q_1)=c^*_0$.
    \item[(A7)]  there exists  
     $q_2>q_1$ such that $c_{NL}^*(q_2)>c^*_0$.
\end{itemize}

\begin{remark}\label{rk:a6+a7}
In view of the assumption  (A3), a simple comparison yields that $c_{NL}^*(\hat{q})\geq c_{NL}^*(q)$ if $\hat{q}\geq q$.
Together with assumptions (A6), (A7) and the fact 
$c_{NL}^*(q)\geq c^*_0\ \ \text{for all}\ \ q\geq0,$
we see that
$$c_{NL}^*(q)=c^*_0\ \ \text{for all}\ \ 0\leq q\leq q_1\ \ \text{and}\ \ c_{NL}^*(q)>c^*_0\ \ \text{for all}\ \ q\geq q_2.$$
\end{remark}

It has been proved in \cite{Carr Chmaj} by Ikehara's Theorem that, if $f(w;q)$ satisfies the KPP condition \eqref{kpp condition}, then
\bea\label{decay-U-linear-1}
\mathcal{W}_q(\xi)=A\xi e^{-\lambda_0\xi}+Be^{-\lambda_0\xi}+o(e^{-\lambda_0\xi})\quad \mbox{as $\xi\to+\infty$},
\eea
where $A>0$ and $B\in \mathbb{R}$.
This asymptotic estimate 
has been extended to the general monostable case ({\it i.e.}, the assumption (A1)) 
with $A\geq0$ and $B\in \mathbb{R}$,
and $B>0$ if $A=0$.
We remark that \eqref{decay-U-linear-1}  has been discussed in
\cite{Coville}; however, the proof provided in \cite[Theorem 1.6]{Coville} contains a gap such that they deduced that $A>0$ always holds in \eqref{decay-U-linear-1}.
We will fix the gap in Proposition~\ref{prop:correction-U-linear-decay} below.


The first result is concerned with how the {\em pulled front} 
evolves to the pulled-to-pushed transition front in terms of the varying parameter $q$. Similar to Theorem \ref{th: threshold scalar equation}, the key point is to completely characterize the evolution of the decay rate of the minimal traveling wave $\mathcal{W}_q(\xi)$ with respect to $q$.

\begin{theorem}\label{th: threshold scalar equation nonlocal}
Assume that assumptions (A1)-(A3) and (A6)-(A7) hold. Then there exists the threshold value $q^*\in[q_1,q_2)$ such that the minimal traveling wave speed of \eqref{scalar nonlocal tw-parameter s} satisfies
\begin{equation}\label{def of threshold scalar nonlocal}
c_{NL}^*(q)=c_0^*\quad\text{for all}\ q\in[0,q^*];\quad c_{NL}^*(q)>c_0^*\quad\text{for all}\ q\in(q^*,\infty).
\end{equation}
Moreover,
the minimal traveling wave $U_s(\xi)$ satisfies 
\begin{equation}\label{asy tw threshold scalar nonlocal}
\mathcal{W}_q(\xi)=Be^{-\lambda_0\xi}+o(e^{-\lambda_0\xi})\quad \mbox{as}\quad \xi\to+\infty\quad\text{for some}\quad B>0,
\end{equation}
if and only if $q=q^*$.
\end{theorem}

\begin{remark}
Note that, \eqref{asy tw threshold scalar nonlocal} in Theorem~\ref{th: threshold scalar equation nonlocal} indicates that, as $\xi\to+\infty$, the leading order of the decay rate of $\mathcal{W}_q(\xi)$ switches from $\xi e^{-\lambda_0\xi}$ to $e^{-\lambda_0\xi}$ as $q\to q^*$ from below.
\end{remark}

\begin{remark}
In our proof of \eqref{def of threshold scalar nonlocal} and the sufficient condition for \eqref{asy tw threshold scalar nonlocal},
the condition in the assumption  (A3) that $f''(0;\hat{q})> f''(0;q)$ for $\hat{q}>q$ is not required.
\end{remark}

\begin{remark}
The classification of traveling wave fronts for \eqref{scalar nonlocal tw-parameter s} has not been completely understood yet. Specifically, when $c>c^*_{NL}$, it is not easy to determine whether traveling wave front decay with the fast order $\lambda^+(c)$ or the slow order $\lambda^-(c)$, where $\lambda^{\pm}(c)$ are defined as that in Remark \ref{rm:lambda_0} but independent on $q$. This open problem will be studied in our forthcoming paper. 
\end{remark}


The rest of this paper is organized as follows. 
In Section 2, we extend our argument to the single reaction-diffusion equation and complete the proof of Theorem \ref{th: threshold scalar equation}. In Section 3, we extend our analysis to the single nonlocal diffusion equation and complete the proof of Theorem \ref{th: threshold scalar equation nonlocal}. The proof for Theorem \ref{th: threshold scalar equation nonlocal} is more involved since the minimal traveling wave speed can not be computed explicitly, but is given by a variational formula.

\section{Threshold of the reaction-diffusion equation}\label{sec:threshold scalar}

In this section, we aim to prove Theorem~\ref{th: threshold scalar equation}.
First, it is well known that for each $s\geq0$, under the assumption  (A1), the minimal traveling wave is unique (up to a translation).
Together with the assumption (A2), one can use the standard compactness argument to conclude that $c^*(s)$ is continuous for all $s\ge 0$.
It follows from assumptions (A3)-(A5) and Remark~\ref{rk:a4+a5} that $c^*(s)$ is nondecreasing in $s$. Thus, we immediately obtain the following result.

\begin{lemma}\label{lem: th1-part1}
Assume that assumptions (A1)-(A5) hold. Then there exists a threshold $s^*\in[s_1,s_2)$ such that \eqref{def of threshold scalar} holds.
\end{lemma}

Thanks to Lemma~\ref{lem: th1-part1}, to prove Theorem~\ref{th: threshold scalar equation}, it suffices to show that
\eqref{asy tw threshold scalar} holds if and only if $s=s^*$.
Let $W_{s^*}$ be the minimal traveling wave satisfying \eqref{scalar tw-parameter s} with $s=s^*$
and $c^*(s^*)=2$. For simplicity, we denote $W_*:=W_{s^*}$.
The first and the most involved step is to show that if $s=s^*$, then \eqref{asy tw threshold scalar} holds.
To do this, we shall use a contradiction argument.
Assume that \eqref{asy tw threshold scalar} is not true. Then, 
it holds that (cf. \cite{Aronson Weinberger})
\bea\label{AS-U-infty-for-contradiction scalar}
\lim_{\xi\rightarrow+\infty}\frac{W_{*}(\xi)}{\xi e^{-\xi}}=A_0\quad \mbox{for some $A_0>0$.}
\eea


Under the condition \eqref{AS-U-infty-for-contradiction scalar},
we shall prove the following proposition.

\begin{proposition}\label{Prop-supersol}
Assume that assumptions (A1)-(A5) hold. In addition, if \eqref{AS-U-infty-for-contradiction scalar} holds,
then there exists an auxiliary continuous function $R_w(\xi)$ defined in $\mathbb{R}$ satisfying
\begin{equation}\label{decay rate of Rw}
R_w(\xi)=O(\xi e^{-\xi}) \quad \mbox{as $\xi\to\infty$},
\end{equation}
such that 
$$\bar{W}(\xi):=\min\{W_{*}(\xi)-R_w(\xi),1\}\geq (\not\equiv) 0$$
is a super-solution
satisfying
\bea\label{tw super solution scalar}
&N_0[\bar W]:=\bar W''+2\bar W'+f(\bar W;s^*+\delta_0)\le0,\quad  \mbox{a.e. in $\mathbb{R}$},
\eea
for some small $\delta_0>0$,
where ${\bar W}'(\xi_0^{\pm})$ exists and
${\bar W}'(\xi_0^+)\leq {\bar W}'(\xi_0^-)$
if $\bar W'$ is not continuous at $\xi_0$.

\end{proposition}

Next, we shall go through a lengthy process to prove Proposition~\ref{Prop-supersol}.
Hereafter, assumptions (A1)-(A5) are always assumed.



From the assumption (A1), by shifting the coordinates, we can immediately obtain
the following lemma.

\begin{lemma}\label{lm: divide R to 3 parts}
Let $\nu_{1}>0$ be 
an arbitrary constant. Then there exist
$$-\infty<\xi_2<0<\xi_1<+\infty\ \ \text{with}\ \ |\xi_1|,|\xi_2|\ \text{very large},$$
such that
the following hold:
\begin{itemize}
    \item[(1)]
    $\dps f(W_{*}(\xi);s^*)=W_{*}(\xi)+\frac{f''(0;s^*)}{2}W^2_{*}(\xi)+o(W^2_{*}(\xi))$ for all $\xi\in[\xi_1,\infty)$;
    \item[(2)]  $f'(W_{*}(\xi);s^*)<0$ for all $\xi\in(-\infty,\xi_2]$.
\end{itemize}
\end{lemma}

\subsection{Construction of the super-solution}\label{subsec-2-1}

\noindent

Let us define $R_w(\xi)$ as (see Figure \ref{Figure scalar})
\begin{equation}\label{definition of Rw}
R_w(\xi)=\begin{cases}
\varepsilon_1\sigma(\xi) e^{-\xi},&\ \ \mbox{for}\ \ \xi\ge\xi_{1}+\delta_1,\\
\varepsilon_2 e^{\lambda_1\xi},&\ \ \mbox{for}\ \ \xi_{2}+\delta_2\le\xi\le\xi_{1}+\delta_1,\\
\varepsilon_3\sin(\delta_4(\xi-\xi_2)),&\ \ \mbox{for}\ \ \xi_2-\delta_3\le\xi\le\xi_2+\delta_2,\\
-\varepsilon_4e^{\lambda_2\xi},&\ \ \mbox{for}\ \ \xi\le\xi_2-\delta_3,
\end{cases}
\end{equation}
where $\delta_{i=1,\cdots,4}>0$, $\lambda_{n=1,2}>0$, and $\sigma(\xi)>0$ will be
determined such that $\bar W(\xi)$ satisfies \eqref{decay rate of Rw} and \eqref{tw super solution scalar}. Moreover, we should choose
positive $\varepsilon_{j=1,\cdots,4}\ll A_0$ ($A_0$ is defined in \eqref{AS-U-infty-for-contradiction scalar}) such that $R_w(\xi)\ll W_*(\xi)$ and $\bar W(\xi)$ is continuous for all $\xi\in\mathbb{R}$.

Since $f(\cdot;s^*)\in C^2$, there exist $K_1>0$ and $K_2>0$ such that
\bea\label{K1K2}
|f''(W_*(\xi);s^*)|< K_1,\quad |f'(W_*(\xi);s^*)|< K_2\quad \mbox{for all}\quad\xi\in\mathbb{R}.
\eea
We set $\lambda_1>0$ 
large enough such that
\begin{equation}\label{condition on lambda 1}
-2\lambda_1-\lambda_1^2+K_2<0\ \text{and}\ \lambda_1>K_2.
\end{equation}
Furthermore, there exists $K_3>0$ such that
\begin{equation}\label{condition on xi_2 scalar}
f'(W_*(\xi); s^*)\le -K_3<0\quad\text{for all}\quad\xi\le \xi_2.
\end{equation}
We set 
$$0<\lambda_2<\lambda_w:=\sqrt{1-f'(1;s^*)}-1$$
sufficiently small
such that
\bea\label{lambda 2}
\lambda_2^2+2\lambda_2-K_3<0.
\eea


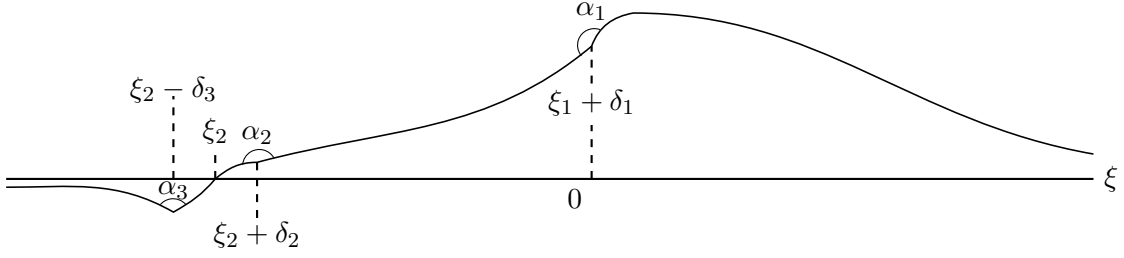
\begin{figure}
\begin{center}
\begin{tikzpicture}[scale = 1.1]
\draw[thick](-6,0) -- (7,0) node[right] {$\xi$};
\draw [semithick] (-6, -0.1) to [ out=0, in=150] (-4,-0.4) to [out=30, in=230] (-3.5,0)  to [out= 40, in=180] (-3,0.2) to [out=15, in=220] (1,1.6) to [out=70,in=190] (1.5,2) to [out=0,in=170] (7,0.3);
\node[below] at (1,1.2) {$\xi_{1}+\delta_1$};
\node[below] at (0.8,0) {$0$};
\draw[dashed] [thick] (-3.5,0)-- (-3.5,0.3);
\node[above] at (-3.5,0.3) {$\xi_2$};
\draw[dashed] [thick] (-3,0.2)-- (-3,-0.5);
\draw[dashed] [thick] (-4,0)-- (-4,1);
\node[above] at (-4,0.8) {$\xi_2-\delta_3$};
\draw[dashed] [thick] (1,0)-- (1,0.65);
\draw[dashed] [thick] (1,1.15)-- (1,1.6);
\node[below] at (-3,-0.4) {$\xi_2+\delta_2$};
\draw [thin] (-3.85,-0.31) arc [radius=0.2, start angle=40, end angle= 140];
\node[above] at (-4,-0.36) {$\alpha_3$};
\draw [thin] (-2.8,0.255) arc [radius=0.2, start angle=30, end angle= 175];
\node[above] at (-3,0.3) {$\alpha_2$};
\draw [thin] (1.1,1.8) arc [radius=0.2, start angle=70, end angle= 220];
\node[above] at (1,1.8) {$\alpha_1$};
\end{tikzpicture}
\caption{the construction of $R_w(\xi)$.}\label{Figure scalar}
\end{center}
\end{figure}


\medskip

We now divide the proof into several steps.

\noindent{\bf{Step 1}:} We consider $\xi\in[\xi_1+\delta_1,\infty)$ where $\delta_1>0$ is small enough and will be determined in Step 2.
In this case, we have
\beaa
R_w(\xi)=\varepsilon_1\sigma(\xi)\ e^{-\xi}
\eeaa
for some small $\varepsilon_1\ll A_0$ 
such that $\bar W=W_*-R_w>0$ for $\xi\geq \xi_1+\delta_1$.

Note that $W_*$ satisfies \eqref{scalar tw-parameter s} with $c=2$.
By some  straightforward computations, we have
\begin{equation}\label{N0 inequality step 1}
\begin{aligned}
N_0[\bar W]=&-R''_w-2 R'_w-f(W_*;s^*)+f(W_*-R_w;s^*+\delta_0)\\
=&-R''_w-2 R'_w-f(W_*;s^*)+f(W_*-R_w;s^*)\\
&-f(W_*-R_w;s^*)+f(W_*-R_w;s^*+\delta_0).
\end{aligned}
\end{equation}
By the assumption (A1) and the statement (1) of Lemma \ref{lm: divide R to 3 parts}, since $W_*\ll 1$ and $R_w\ll W_*$ for $\xi\in[\xi_1+\delta_1,\infty)$, we have
\begin{equation}\label{estimate 1}
-f(W_*;s^*)+f(W_*-R_w;s^*)= -R_w+f''(0;s^*)(\frac{R^2_w}{2}-W_*R_w)+o((W_*)^2).
\end{equation}
By the assumption (A2) and the statement (1) of Lemma \ref{lm: divide R to 3 parts}, there exists $C_1>0$ such that
\begin{equation}\label{estimate 2}
-f(W_*-R_w;s^*)+f(W_*-R_w;s^*+\delta_0)\le C_1\delta_0(W_*-R_w)^2+o((W_*)^2).
\end{equation}
From \eqref{K1K2}, \eqref{N0 inequality step 1}, \eqref{estimate 1}, \eqref{estimate 2}, 
we have
\bea\label{estimate 3}
N_0[\bar W]\le -\varepsilon_1\sigma''e^{-\xi}+K_1(\frac{R_w^2}{2}+W_*R_w)+C_1\delta_0W_*^2+o((W_*)^2.
\eea

Now, we define
$$\sigma(\xi):=4e^{-\frac{1}{2}(\xi-\xi_1)}-4+4\xi-4\xi_1$$
which satisfies 
$$\sigma(\xi_1)=0,\ \sigma'(\xi)=4-2e^{-\frac{1}{2}(\xi-\xi_1)},\ \sigma''(\xi)=e^{-\frac{1}{2}(\xi-\xi_1)}.$$
Moreover, $\sigma(\xi)= O(\xi)$ as $\xi\to\infty$ implies that $R_w$ satisfies \eqref{decay rate of Rw}.

Due to \eqref{AS-U-infty-for-contradiction scalar} and the equation of $W_*$, we may also assume
\bea\label{W-M}
W_*(\xi) \leq 2A_0\xi e^{-\xi}\quad \mbox{for all}\quad\xi\geq \xi_1.
\eea
Then, from \eqref{estimate 3}, up to enlarging $\xi_1$ if necessary, we always have 
\beaa
N_0[\bar W]\le -\varepsilon_1e^{-\frac{1}{2}(\xi-\xi_1)}e^{-\xi}+K_1(\frac{R_w^2}{2}+W_*R_w)+C_1\delta_0W_*^2+o((W_*)^2)\le 0
\eeaa
for all sufficiently small $\delta_0>0$
since $R_w^2(\xi)$, $W_*R_w(\xi)$, and  $W_*^2(\xi)$ are $o(e^{-\frac{3}{2}\xi})$
as $\xi\to\infty$
from \eqref{W-M} and the definition of $R_w$.
Therefore, $N_0[\bar W]\leq 0$ for $\xi\ge \xi_1$.

\medskip

\noindent{\bf{Step 2:}} We consider $\xi\in[\xi_2+\delta_2,\xi_1+\delta_1]$ for some small $\delta_2>0$, and small $\delta_1>0$ satisfying
\bea\label{cond delta 1 scalar}
1+3(1-e^{-\frac{\delta_1}{2}})-2\delta_1>0.
\eea
From the definition of $R_w$ in Step 1, it is easy to check that $R'_w((\xi_1+\delta_1)^+)>0$ under the condition \eqref{cond delta 1 scalar}.
In this case, we have $R_w(\xi)=\varepsilon_2 e^{\lambda_1\xi}$
for some large $\lambda_1>0$ satisfying \eqref{condition on lambda 1}.

We first choose
\begin{equation}\label{condition on epsilon 2-RD}
\varepsilon_2=\varepsilon_1\Big(4e^{-\frac{\delta_1}{2}}-4+4\delta_1\Big)e^{-(1+\lambda_1)(\xi_1+\delta_1)}
\end{equation}
such that $R_w(\xi)$ is continuous at $\xi=\xi_{1}+\delta_1$. Then, from \eqref{condition on epsilon 2-RD}, we have
$$R'_w((\xi_1+\delta_1)^{+})=\varepsilon_1\sigma'(\xi_1+\delta_1)e^{-(\xi_1+\delta_1)}- R_w(\xi_1+\delta_1)>R'_w((\xi_1+\delta_1)^{-})=\lambda_1R_w(\xi_1+\delta_1)$$
is equivalent to
$$1+(3+2\lambda_1)(1-e^{-\frac{\delta_1}{2}})>2(1+\lambda_1)\delta_1,$$
which holds by taking $\delta_1$ sufficiently small.
This implies that $\angle\alpha_1<180^{\circ}$.

By some  straightforward computations, we have
\begin{equation*}
\begin{aligned}
N_0[\bar W]=&-(2\lambda_1+\lambda_1^2)R_w-f(W_*;s^*)+f(W_*-R_w;s^*+\delta_0)\\
=&-(2\lambda_1+\lambda_1^2)R_w-f(W_*;s^*)+f(W_*-R_w;s^*)\\
&-f(W_*-R_w;s^*)+f(W_*-R_w;s^*+\delta_0).
\end{aligned}
\end{equation*}
Thanks to \eqref{K1K2}, we have
$$-f(W_*;s^*)+f(W_*-R_w;s^*)< K_2R_w.$$
Moreover, by assumption (A2), 
$$-f(W_*-R_w;s^*)+f(W_*-R_w;s^*+\delta_0)\le L_0\delta_0.$$
Then, 
since
$\lambda_1$ satisfies \eqref{condition on lambda 1}, we have
$$L_0\delta_0<\varepsilon_2(\lambda_1^2+2\lambda_1-K_2) e^{\lambda_1(\xi_2+\delta_2)}$$
for all sufficiently small $\delta_0>0$,
which implies that $N_0[\bar W]\le 0$ for all $\xi\in[\xi_2+\delta_2,\xi_1+\delta_1]$.

\medskip

\noindent{\bf{Step 3:}} We consider $\xi\in[\xi_2-\delta_3,\xi_2+\delta_2]$ for some small $\delta_2,\delta_3>0$. In this case, $R_w(\xi)=\varepsilon_3\sin(\delta_4(\xi-\xi_2))$.
We first verify the following Claim.
\begin{claim}\label{cl scalar}
For any $\delta_2$ with $\delta_2>\frac{1}{\lambda_1}$, there exist
$\varepsilon_3>0$ and small $\delta_4>0$
such that 
$$R_w((\xi_2+\delta_2)^+)=R_w((\xi_2+\delta_2)^-)$$ 
and $\angle\alpha_2<180^{\circ}$.
\end{claim}
\begin{proof}
Note that 
$$R_w((\xi_2+\delta_2)^+)=\varepsilon_2e^{\lambda_1(\xi_2+\delta_2)},\ R_w((\xi_2+\delta_2)^-)=\varepsilon_3\sin(\delta_4\delta_2).$$
Therefore, we may take
\bea\label{epsilon 3}
\varepsilon_3= \frac{\varepsilon_2e^{\lambda_1(\xi_2+\delta_2)}}{\sin(\delta_4\delta_2)}>0
\eea
such that $R_w((\xi_2+\delta_2)^+)=R_w((\xi_2+\delta_2)^-)$.

By some  straightforward computations, we have
$R'_w((\xi_2+\delta_2)^+)=\lambda_1\varepsilon_2e^{\lambda_1(\xi_2+\delta_2)}$ and
\beaa
R'_w((\xi_2+\delta_2)^-)=\varepsilon_3\delta_4\cos(\delta_4\delta_2)=\frac{\varepsilon_2e^{\lambda_1(\xi_2+\delta_2)}}{\sin(\delta_4\delta_2)}\delta_4\cos(\delta_4\delta_2),
\eeaa
which yields that
$$R'_w((\xi_2+\delta_2)^-)\rightarrow \varepsilon_2e^{\lambda_1(\xi_2+\delta_2)}/\delta_2\ \ \text{as}\ \ \delta_4\to0.$$
In other words, as $\delta_4\to0$,
\bea\label{delta 2}
R'_w((\xi_2+\delta_2)^+)>R'_w((\xi_2+\delta_2)^-)\ \ \text{is equivalent to}\ \ \delta_2>\frac{1}{\lambda_1}.
\eea
Therefore, we can choose $\delta_4>0$ sufficiently small so that $\angle\alpha_2<180^{\circ}$.
This completes the proof of Claim~\ref{cl scalar}.
\end{proof}

Next, we verify the differential inequality of $N_0[\bar W]$ for $\xi\in[\xi_2-\delta_3,\xi_2+\delta_2]$. By some  straightforward computations, we have
\begin{equation*}
\begin{aligned}
N_0[\bar W]=&\delta^2_4R_w-2\varepsilon_3\delta_4\cos(\delta_4(\xi-\xi_2))\\
&-f(W_*;s^*)+f(W_*-R_w;s^*)-f(W_*-R_w;s^*)+f(W_*-R_w;s^*+\delta_0).
\end{aligned}
\end{equation*}
The same argument as in Step 2 implies that
$$-f(W_*;s^*)+f(W_*-R_w;s^*)\le K_2R_w\ \ \text{and}\ \ -f(W_*-R_w;s^*)+f(W_*-R_w;s^*+\delta_0)\le L_0\delta_0,$$
which yields that
\beaa
N_0[\bar W]\leq\delta^2_4R_w-2\varepsilon_3\delta_4\cos(\delta_4(\xi-\xi_2))+K_2R_w+L_0\delta_0.
\eeaa
We first focus on $\xi\in[\xi_2,\xi_2+\delta_2]$. By \eqref{epsilon 3}, \eqref{delta 2}, and the definition of $\lambda_1$ (see \eqref{condition on lambda 1}), we can choose
$\delta_2\in(1/\lambda_1,1/K_2)$ such that
\beaa
\min_{\xi\in[\xi_2,\xi_2+\delta_2]}\delta_4\varepsilon_3\cos(\delta_4(\xi-\xi_2))\to\frac{\varepsilon_2e^{\lambda_1(\xi_2+\delta_2)}}{\delta_2}
=\frac{R_w(\xi_2+\delta_2)}{\delta_2}> K_2R_w(\xi_2+\delta_2)
\quad\text{as}\ \delta_4\to0.
\eeaa
Thus, we have
\beaa
\min_{\xi\in[\xi_2,\xi_2+\delta_2]}\Big[\delta_4\varepsilon_3\cos(\delta_4(\xi-\xi_2))-(K_2+\delta_4^2)R_w(\xi)\Big]>0,
\eeaa
for all sufficiently small $\delta_4>0$.
Then, for all sufficiently small $\delta_0>0$, 
we see that $N_0[\bar W]\le 0$ for $\xi\in[\xi_2,\xi_2+\delta_2]$.

For $\xi\in[\xi_2-\delta_3,\xi_2]$,
by setting $\delta_3>0$ small enough,
$N_0[\bar W]\le 0$ can be verified easier by the same argument since $R_w<0$. This completes the Step 3.

\medskip

\noindent{\bf{Step 4:}} We consider $\xi\in(-\infty,\xi_2-\delta_3]$. In this case, we have $R_w(\xi)=-\varepsilon_4e^{\lambda_2\xi}<0$.
Recall that we choose $0<\lambda_2<\lambda_w$
and 
$$1-W_*(\xi)\sim C_2 e^{\lambda_{w}\xi}\ \ \text{as}\ \ \xi\to-\infty.$$
Then, there exists $M>0$ such that 
$$\bar W:=\min\{W_*-R_w,1\}\equiv 1\ \ \text{for all} \ \ \xi\le -M,$$
and thus $N_0[\bar W]\le 0$ for all $\xi\le -M$. Therefore, we only need to show 
$$N_0[\bar W]\le 0\ \ \text{for all}\ \ -M\le\xi\le -\xi_2-\delta_3.$$

We first choose 
$$\varepsilon_4=\varepsilon_3\sin(\delta_4\delta_3)/e^{\lambda_2(\xi_2-\delta_3)}$$
such that $R_w$ is continuous at $\xi_2-\delta_3$. It is easy to check that 
$$R'_w((\xi_2-\delta_3)^+)>0>R'_w((\xi_2-\delta_3)^-),$$
 and hence $\angle\alpha_3< 180^{\circ}$.

By some  straightforward computations, we have
\begin{equation*}
\begin{aligned}
N_0[\bar W]=&-(\lambda_2^2+2\lambda_2)R_w-f(W_*;s^*)+f(W_*-R_w;s^*+\delta_0)\\
=&-(\lambda_2^2+2\lambda_2)R_w-f(W_*;s^*)+f(W_*-R_w;s^*)\\
&-f(W_*-R_w;s^*)+f(W_*-R_w;s^*+\delta_0).
\end{aligned}
\end{equation*}
From \eqref{condition on xi_2 scalar}, we have 
$$-f(W_*;s^*)+f(W_*-R_w;s^*)< K_3R_w<0.$$
Together with the assumption  (A2), we have
\beaa
N_0[\bar W]\leq -(\lambda_2^2+2\lambda_2-K_3)R_w+L_0\delta_0\quad\text{for all}\quad \xi\in[-M,\xi_2-\delta_3].
\eeaa
In view of  \eqref{lambda 2},
we can assert that 
$$N_0[\bar W]\le 0\ \ \text{for all}\ \ \xi\in[-M,\xi_2-\delta_3],$$
 provided that $\delta_0$ is sufficiently small.
This completes the Step 4.

\subsection{Proof of Theorem \ref{th: threshold scalar equation}}

\noindent

We first complete the proof of Proposition~\ref{Prop-supersol}.
\begin{proof}[Proof of Proposition~\ref{Prop-supersol}]
From the discussion from Step 1 to Step 4 in \S \ref{subsec-2-1},
we are now equipped with a suitable function $R_w(\xi)$ 
defined as in \eqref{definition of Rw} such that 
$$\bar W (\xi)=\min \{W_*(\xi)-R_w(\xi),1\},$$
 which is independent of the choice of all sufficiently small $\delta_0>0$, forms
a super-solution satisfying \eqref{tw super solution scalar}.
Therefore, we complete the proof of Proposition~\ref{Prop-supersol}.
\end{proof}

Now, we are ready to prove Theorem~\ref{th: threshold scalar equation} as follows.

\begin{proof}[Proof of Theorem~\ref{th: threshold scalar equation}]
In view of Lemma~\ref{lem: th1-part1}, we have obtained \eqref{def of threshold scalar}.
It suffices to show that \eqref{asy tw threshold scalar} holds if and only if $s=s^*$.
First, we show that
\bea\label{th1:goal-1}
s=s^* \quad \Longrightarrow \quad \mbox{\eqref{asy tw threshold scalar} holds}.
\eea
Suppose that \eqref{asy tw threshold scalar} does not hold. Then $W_*$ satisfies \eqref{AS-U-infty-for-contradiction scalar}. In view of Proposition~\ref{Prop-supersol}, we can
choose $\delta_0>0$ sufficiently small such that
$$\bar W(\xi)= \min\{W_*(\xi)-R_w(\xi),1\}\geq(\not\equiv) 0$$
satisfies \eqref{tw super solution scalar}. Next,
we consider the following Cauchy problem with compactly supported initial datum $0\le w_0(x)\leq \bar W(x)$:
\begin{equation}\label{scalar cauchy problem}
\left\{
\begin{aligned}
&w_t=w_{xx}+f(w;s^*+\delta_0),\ t\ge 0,\ x\in\mathbb{R},\\
&w(0,x)=w_0(x),\ x\in\mathbb{R}.
\end{aligned}
\right.
\end{equation}
Then, in view of \eqref{def of threshold scalar}, we see that $c^*(s^*+\delta_0)>2$ (the minimal speed is nonlinearly selected). Therefore, we can apply Theorem 2 of \cite{Rothe1981}
to conclude that the spreading speed of the Cauchy problem \eqref{scalar cauchy problem} is strictly greater than $2$.

On the other hand, we
define $\bar{w}(t,x):=\bar W(x-2 t)$, and hence 
$$\bar{w}(0,x)=\bar W(x)\geq w_0(x)\ \ \text{for all}\ \ x\in\mathbb{R}.$$
Since $\bar W$ satisfies \eqref{tw super solution scalar},
$\bar{w}$ forms a super-solution of \eqref{scalar cauchy problem}.
This immediately implies that the spreading speed of the solution, 
namely $w(t,x)$, of \eqref{scalar cauchy problem} is slower than or equal to $2$, due to the comparison principle.
This contradicts the spreading speed of the Cauchy problem \eqref{scalar cauchy problem}, which is strictly greater than $2$.
Thus, we obtain \eqref{th1:goal-1}.


Finally, we prove that
\bea\label{th1:goal-2}
\mbox{\eqref{asy tw threshold scalar} holds} \quad \Longrightarrow \quad s=s^*.
\eea
Note that for $s>s^*$, from \eqref{def of threshold scalar} we see that $c^{*}(s)>2$; so the asymptotic behavior of $W_s$ at $\xi\approx+\infty$ in Proposition \ref{prop: classification scalar} implies that
\eqref{asy tw threshold scalar} does not hold for any $s>s^*$. Therefore,
we only need to show that if $s<s^*$, then \eqref{asy tw threshold scalar} does not hold.
We assume by contradiction that there exists $s_0\in(0,s^*)$ such that
the corresponding minimal traveling wave satisfies
\bea\label{W-S0+infty}
W_{s_0}(\xi)=B_0 e^{-\xi}+o(e^{-\xi})\quad\text{as}\quad \xi\to+\infty
\eea
for some $B_0>0$. For $\xi\approx -\infty$, we have
\bea\label{W-S0-infty}
1-W_{s_0}(\xi)=C_0 e^{\hat\lambda\xi}+o(e^{\hat\lambda\xi})\quad\text{as}\quad \xi\to-\infty
\eea
for some $C_0>0$, where $\hat\lambda:=\sqrt{1-f'(1;s_0)}-1$.
Recall that the asymptotic behavior of $W_s^*$ at $\pm\infty$ satisfies
\begin{equation}\label{W-S*-pm-infty}
\begin{aligned}
W_{s^*}(\xi)=B e^{-\xi}+o(e^{-\xi})\ \text{as}\ \xi\to+\infty,\\
1-W_{s^*}(\xi)=C e^{\lambda_w\xi}+o(e^{\lambda_w\xi})\ \text{as}\ \xi\to-\infty,
\end{aligned}
\end{equation}
for some $B,C>0$, where $\lambda_w=\sqrt{1-f'(1;s^*)}-1$. In view of the assumption (A3), we have $\lambda_w> \hat\lambda$.
Combining \eqref{W-S0+infty}, \eqref{W-S0-infty}, and \eqref{W-S*-pm-infty}, there exists $L>0$ sufficiently large such that
$$W_{s^*}(\xi-L)> W_{s_0}(\xi)\ \ \text{for all}\ \ \xi\in\mathbb{R}.$$
Now, we define
\beaa
L^*:=\inf\{L\in\mathbb{R}\ |\ W_{s^*}(\xi-L)\ge W_{s_0}(\xi)\ \text{for all}\ \xi\in\mathbb{R}\}.
\eeaa
By the continuity, we have 
$$W_{s^*}(\xi-L^*)\geq W_{s_0}(\xi)\ \ \text{for all}\ \ \xi\in\mathbb{R}.$$
 If there exists $\xi^*\in\mathbb{R}$ such that
$W_{s^*}(\xi^*-L^*)= W_{s_0}(\xi^*)$, by the strong maximum principle, we have $W_{s^*}(\xi-L^*)=W_{s_0}(\xi)$ for $\xi\in\mathbb{R}$,
which is impossible since $W_{s^*}(\cdot-L^*)$ and $W_{s_0}(\cdot)$ satisfy different equations. Consequently,
$$W_{s^*}(\xi-L^*)> W_{s_0}(\xi)\ \ \text{for all}\ \ \xi\in\mathbb{R}.$$
In particular, we have
\beaa
\lim_{\xi\to\infty}\frac{W_{s^*}(\xi-L^*)}{W_{s_0}(\xi)}\geq1.
\eeaa
Furthermore, we can claim that
\bea\label{limit=1}
\lim_{\xi\to\infty}\frac{W_{s^*}(\xi-L^*)}{W_{s_0}(\xi)}=1.
\eea
Otherwise, if the limit in \eqref{limit=1} is strictly bigger than 1, together with
\beaa
\lim_{\xi\to-\infty}\frac{1-W_{s^*}(\xi-L^*)}{1-W_{s_0}(\xi)}=0,
\eeaa
we can easily  find $\varepsilon>0$ sufficiently small such that
$$W_{s^*}(\xi-(L^*+\varepsilon))> W_{s_0}(\xi)\ \ \text{for}\ \ \xi\in\mathbb{R},$$
which contradicts the definition of $L^*$.
As a result, from \eqref{W-S0+infty}, \eqref{W-S*-pm-infty} and \eqref{limit=1}, we obtain $B_0=Be^{L^*}$.

On the other hand, we set $\widehat{W}(\xi)=W_{s^*}(\xi-L^*)-W_{s_0}(\xi)$. Then $\widehat{W}(\xi)$ satisfies
\bea\label{W-hat-eq2}
\widehat{W}''+2\widehat{W}'+\widehat{W}+J(\xi)=0, \quad \xi\in\mathbb{R},
\eea
where
$$J(\xi)=f(W_{s^*};s^*)- W_{s^*}-f(W_{s_0};s_0)+ W_{s_0}.$$
By the assumption  (A1) and Taylor's Theorem, there exist $\eta_1\in(0, W_{s^*})$ and $\eta_2\in(0,W_{s_0})$ such that
\beaa
J(\xi)&=&f''(\eta_1;s^*)W^2_{s^*}-f''(\eta_2;s_0)W^2_{s_0}\\
      &=&f''(\eta_1;s^*)(W^2_{s^*}-W^2_{s_0})+[f''(\eta_1;s^*)-f''(\eta_2;s_0)]W^2_{s_0}\\
      &=&f''(\eta_1;s^*)(W_{s^*}+W_{s_0})\widehat{W}+[f''(\eta_1;s^*)-f''(\eta_2;s_0)]W^2_{s_0}.
\eeaa
Define 
$$J_1(\xi):=f''(\eta_1;s^*)(W_{s^*}+W_{s_0})\widehat{W},$$ 
$$J_2(\xi):=[f''(\eta_1;s^*)-f''(\eta_2;s_0)]W^2_{s_0}.$$
 It is easy to see that $J_1(\xi)=o(\widehat{W})$ for $\xi\approx+\infty$. Next, we will show $J_2(\xi)=o(\widehat{W})$ for $\xi\approx+\infty$.

Since $f''(0;s^*)>f''(0;s_0)$ (from the assumption (A3)), we can find small $\delta>0$ such that
$$\min_{\eta\in[0,\delta]}f''(\eta;s^*)>\max_{\eta\in[0,\delta]}f''(\eta;s_0)$$
and thus
there exist $\kappa_1,\kappa_2>0$ such that
\bea\label{J-lower bound}
\kappa_1e^{-2\xi}\ge J_2(\xi)=[f''(\eta_1;s^*)-f''(\eta_2;s_0)]W^2_{s_0}(\xi)\ge \kappa_2 e^{-2\xi}\quad \mbox{for all large $\xi$}.
\eea

We now claim that $J_2(\xi)=o(\widehat{W})$ as $\xi\to+\infty$.
For contradiction, we assume that it is not true. Then there exists $\{\xi_n\}$ with
$\xi_n\to+\infty$ as $n\to\infty$ such that for some $\kappa_3>0$,
\bea\label{kappa3}
\frac{J_2(\xi_n)}{\widehat{W}(\xi_n)}\geq \kappa_3\quad \mbox{for all $n\in\mathbb{N}$.}
\eea
Set $\widehat{W}(\xi)=\alpha(\xi)e^{-2\xi}$, where $\alpha(\xi)>0$ for all $\xi$.
By substituting it into \eqref{W-hat-eq2},
we have 
\bea\label{alpha-eq}
L(\xi):=(\alpha''(\xi)-2\alpha'(\xi)+\alpha(\xi))e^{-2\xi}+J_1(\xi)+J_2(\xi)=0\quad \mbox{for all large $\xi$}.
\eea
By \eqref{J-lower bound} and \eqref{kappa3}, we have 
\bea\label{alpha-bdd}
0<\alpha(\xi_n)\leq \frac{\kappa_1}{\kappa_3}\quad  \mbox{for all $n\in\mathbb{N}$.}
\eea
Now, we will reach a contradiction by dividing the behavior of $\alpha(\cdot)$ into two cases:
\begin{itemize}
    \item[(i)] $\alpha(\xi)$ oscillates for all large $\xi$;
    \item[(ii)] $\alpha(\xi)$ is monotone for all large $\xi$.
\end{itemize}
For case (i), there exist local minimum points $\eta_n$ of $\alpha$ with $\eta_n\to\infty$ as $n\to\infty$ such that
\beaa
\alpha(\eta_n)>0,\quad \alpha'(\eta_n)=0,\quad \alpha''(\eta_n)\geq 0\quad  \mbox{for all $n\in\mathbb{N}$.}
\eeaa
Together with \eqref{J-lower bound} and $J_1(\xi)=o(\hat W(\xi))$,
from \eqref{alpha-eq} we see that
\beaa
0=L(\eta_n)\geq \alpha(\eta_n)e^{-2\eta_n}+o(1)\alpha(\eta_n)e^{-2\eta_n}+\kappa_2e^{-2\eta_n}>0
\eeaa
for all large $n$, which reaches a contradiction. 

For case (ii),
due to \eqref{alpha-bdd}, there exists $\alpha_0\in[0, \kappa_1/\kappa_3]$
such that $\alpha(\xi)\to \alpha_0$ as $\xi\to\infty$. Hence, we can find subsequence $\{\eta_j\}$ that tends to $\infty$ such that $\alpha'(\eta_j)\to0$, $\alpha''(\eta_j)\to0$ and
$\alpha(\eta_j)\to \alpha_0$ as $n\to\infty$.
From \eqref{alpha-eq} we deduce that
\beaa
0=L(\eta_j)\geq (o(1)+\alpha(\eta_j)+ \kappa_2)e^{-2\eta_j}>0
\eeaa
for all large $j$, which reaches a contradiction. 
Therefore, we have proved that
$J_2(\xi)=o(\widehat{W})$ as $\xi\to\infty$.
Consequently, we have
\bea\label{J-small o}
J(\xi)=J_1(\xi)+J_2(\xi)=o(\widehat{W}(\xi))\quad \mbox{as $\xi\to\infty$.}
\eea

Thanks to \eqref{J-small o}, we can apply
\cite[Chapter 3, Theorem 8.1]{CoddingtonLevison} to assert that
the asymptotic behavior of $\widehat{W}(\xi)$ at $\xi=+\infty$ satisfies
\beaa
\widehat{W}(\xi)=(C_1\xi+C_2)e^{-\xi}+o(e^{-\xi})\quad \mbox{as $\xi\to\infty$},
\eeaa
where  $C_1\geq0$, and $C_2>0$ if $C_1=0$. From \eqref{W-S0+infty} and \eqref{W-S*-pm-infty}, we see that $C_1=0$, and $C_2>0$. On the other hand, $B_0=Be^{ L^*}$ implies that $C_2=0$, which reaches a contradiction.
Therefore, \eqref{th1:goal-2} holds, and the proof is complete.
\end{proof}


\section{Threshold of the nonlocal diffusion equation}\label{sec: threshold scalar nonlocal}

In this section, we aim to prove Theorem~\ref{th: threshold scalar equation nonlocal}. 
The main idea is similar to that we used for Theorem \ref{th: threshold scalar equation}. The most involved part is how to construct a suitable super-solution to get the contradiction.

\subsection{Preliminary}
We first introduce some propositions concerned with the asymptotic behavior of the minimal traveling wave of \eqref{scalar nonlocal tw-parameter s} as $\xi\to+\infty$ and $\xi\to-\infty$.
To obtain the asymptotic behavior at $\xi\to+\infty$, we will use specific linearized results established in \cite{Chen Fu Guo, Zhang_etal2012} and a modified version of Ikehara's Theorem (see Proposition 2.3 in \cite{Carr Chmaj}).
\begin{proposition}[Proposition 3.7 in \cite{Zhang_etal2012}]\label{prop:Zhang-etal}
Assume that $c>0$ and $B(\cdot)$ is a continuous function having finite limits at infinity $B(\pm \infty):=\lim_{\xi\to\pm\infty}B(\xi)$.
Let $z(\cdot)$ be a measurable function satisfying
\beaa
c z(\xi)=\int_{\mathbb{R}}J(y)e^{\int_{\xi-y}^{\xi}z(s)ds}dy+B(\xi),\quad \xi\in\mathbb{R}.
\eeaa
Then $z$ is uniformly continuous and bounded. Furthermore, $\omega^{\pm} = \lim_{\xi\to\pm\infty} z(\xi)$ exist and are real roots of the
characteristic equation
\beaa
c\omega=\int_{\mathbb{R}}J(y)e^{\omega y}dy+B(\pm\infty).
\eeaa
\end{proposition}

\begin{proposition}[Ikehara's Theorem]\label{prop:ikehara}
For a positive non-increasing function $U$, we define
\begin{eqnarray*}
\dps F(\lambda):=\int_0^{+\infty} e^{-\lambda\xi} U(\xi) d\xi,\quad
\mbox{$\lambda\in\mathbb{C}$ with ${\rm Re}\lambda<0$}.
\end{eqnarray*}
If $F$ can be written as
$F(\lambda)={H(\lambda)}/{(\lambda+\gamma)^{p+1}}$
for some constants $p>-1, \gamma>0$, and some analytic function $H$
in the strip $-\gamma\leq {\rm Re}\lambda<0$, then
\begin{eqnarray*}
\lim_{\xi\rightarrow +\infty}
\frac{U(\xi)}{{\xi}^{p}e^{-\gamma\xi}}=\frac{H(-\gamma)}{\Gamma(\gamma+1)}.
\end{eqnarray*}
\end{proposition}

\begin{proposition}\label{prop:correction-U-linear-decay}
Assume that $c=c_{NL}^*(q)=c^*_0$. Let $\lambda_0$ be defined as that in Remark~\ref{rm:lambda_0}.  Then the minimal traveling wave $\mathcal{W}_q(\xi)$ satisfies
\bea\label{decay-U-linear}
\mathcal{W}_q(\xi)=A\xi e^{-\lambda_0\xi}+Be^{-\lambda_0\xi}+o(e^{-\lambda_0\xi})\quad \mbox{as $\xi\to+\infty$},
\eea
where $A\geq0$ and $B\in \mathbb{R}$,
and $B>0$ if $A=0$.
\end{proposition}
\begin{proof}
For convenience, we write $\mathcal{W}$ instead of $\mathcal{W}_q(\xi)$.
Let $z(\xi):=-\mathcal{W}'(\xi)/\mathcal{W}(\xi)$. Then, from \eqref{scalar nonlocal tw-parameter s} we have
\beaa
cz(\xi)=\int_{\mathbb{R}}J(y)e^{\int_{\xi-y}^{\xi}z(s)ds}dy+B(\xi),
\eeaa
where $B(\xi)=f(\mathcal{W})/\mathcal{W}-1$.
Since $\mathcal{W}(+\infty)=0$, we have $B(+\infty)=f'(0)-1$.
It follows from Proposition~\ref{prop:Zhang-etal} and Remark~\ref{rm:lambda_0} that
\bea\label{limit-z}
\lim_{\xi\to+\infty}\frac{\mathcal{W}'(\xi)}{\mathcal{W}(\xi)}=-\lim_{\xi\to+\infty} z(\xi)=-\lambda_0.
\eea

With \eqref{limit-z}, we can correct the proof of \cite[Theorem 1.6]{Coville} and obtain the desired result. To see this, we set 
\bea\label{mathcal F}
\mathcal{F}(\lambda)=\int_0^{\infty}\mathcal{W}(\xi) e^{-\lambda \xi}d\xi.
\eea
Because of \eqref{limit-z},
$\mathcal{F}$ is well-defined for $\lambda\in\mathbb{C}$ with $-\lambda_0<{\rm Re}\lambda<0$.
From \eqref{scalar nonlocal tw-parameter s}, we can rewrite it as
\beaa
(c\lambda+h(\lambda))\int_{\mathbb{R}}\mathcal{W}(\xi)e^{-\lambda \xi}d\xi=
\int_{\mathbb{R}} e^{-\lambda \xi}[f'(0)\mathcal{W}(\xi)-f(\mathcal{W}(\xi))]d\xi=:Q(\lambda),
\eeaa
where $h(\lambda)=h(-\lambda)$ is defined in Remark~\ref{rm:lambda_0}.
Moreover, we see that $Q(\lambda)$ is well-defined for $\lambda\in\mathbb{C}$ with
$-2\lambda_0<{\rm Re} \lambda<0$
since
$$f(w)=f'(0)w+O(w^2)\ \ \text{as}\ \ w\to0.$$
Then, we have
\bea\label{F-form-nonlocal}
\mathcal{F}(\lambda)=\frac{Q(\lambda)}{c\lambda+h(\lambda)}-\int_{-\infty}^{0}\mathcal{W}(\xi)e^{-\lambda \xi}d\xi,
\eea
as long as $\mathcal{F}(\lambda)$ is well-defined.

To apply Ikehara's Theorem (Proposition~\ref{prop:ikehara}), we
rewrite \eqref{F-form-nonlocal} as
\beaa
\mathcal{F}(\lambda)=\frac{H(\lambda)}{(\lambda+\lambda_0)^{p+1}},
\eeaa
where $p\in\mathbb{N}\cup\{0\}$ and
\bea\label{Ikehara-form-nonlocal}
H(\lambda)=\frac{Q(\lambda)}{(c\lambda+h(\lambda))/(\lambda+\lambda_0)^{p+1}}
-(\lambda+\lambda_0)^{p+1}\int_{-\infty}^{0}e^{-\lambda \xi}\mathcal{W}(\xi) d\xi.
\eea
It is well known from (cf. \cite[p.2437]{Carr Chmaj}) that all roots of $c\lambda+h(\lambda)=0$ must be real.  Together with the assumption $c_{NL}^*=c^*_0$ and Remark~\ref{rm:lambda_0},
we see that $\lambda = -\lambda_0$ is the only (double) root of $c\lambda+h(\lambda)=0$.

Next, we will show $H$ is analytic in the strip $\{-\lambda_0\leq  {\rm Re}\lambda<0\}$ and $H(-\lambda_0)\neq 0$ with some $p\in\mathbb{N}\cup\{0\}$.
Note that the second term on the right-hand side of \eqref{Ikehara-form-nonlocal} is analytic on $\{{\rm Re}\lambda<0\}$. 
Consequently, it is enough to deal with the first term. 
\begin{itemize}
    \item[(i)] Assume that
$Q(-\lambda_0)\neq 0$. Then
 by setting $p=1$, we obtain $H(-\lambda_0)\neq 0$ (since $c\lambda+h(\lambda)=0$ has the double root $\lambda_0$), and thus
\beaa
\lim_{\xi\to+\infty}\frac{\mathcal{W}(\xi)}{\xi e^{-\lambda_0\xi}}=C_1
\eeaa
for some $C_1>0$ by Ikehara's Theorem (Proposition~\ref{prop:ikehara}).
\item[(ii)] Assume that
$Q(-\lambda_0)=0$. This means that $\lambda=-\lambda_0$ is a root of $Q(\lambda)$.
One can observe from \eqref{F-form-nonlocal} that the root $\lambda=-\lambda_0$ of $Q$ must be simple; otherwise, $\mathcal{F}(\lambda)$ has a removable singularity at $\lambda=-\lambda_0$ and thus can
be extended to exist over $\{ -\lambda_0-\epsilon\leq  {\rm Re} \lambda<0\}$
for some $\epsilon>0$. However, by \eqref{limit-z} and \eqref{mathcal F}, we see that $\mathcal{F}(\lambda)$ is divergent for
$\lambda$ with ${\rm Re} \lambda<-\lambda_0$,
which leads to a contradiction. Therefore, $\lambda=-\lambda_0$ is a simple root of $Q$.
By taking $p=0$ in \eqref{Ikehara-form-nonlocal}, we obtain $H(\lambda_0)\neq 0$, and thus
\beaa
\lim_{\xi\to+\infty}\frac{\mathcal{W}(\xi)}{e^{-\lambda_0\xi}}=C_2
\eeaa
for some $C_2>0$ by Ikehara's Theorem (Proposition~\ref{prop:ikehara}).
\end{itemize}
As a result, we obtain \eqref{decay-U-linear} in which $A$ and $B$ cannot be equal to $0$ at the same time. 
\end{proof}

The third proposition provides the asymptotic behavior of the minimal traveling wave as $\xi\to-\infty$,
\begin{proposition}\label{prop: asy tw - infty nonlocal}
Let $\mathcal{W}_{q,c}$ be the traveling wave satisfying \eqref{scalar nonlocal tw-parameter s} with speed $c\ge c^*_0$ and $q\ge 0$.
We define $\mu_{q,c}$ as the unique positive root of
\begin{equation}\label{nonlocal root}
c\mu=I_1(\mu):=\int_{\mathbb{R}}J(y)e^{\mu y}dy+f'(1;q)-1.
\end{equation}
Then it holds
\begin{equation*}
1-\mathcal{W}_{q,c}(\xi)=O(e^{\mu_{q,c}\xi})\quad\text{as}\quad \xi\to-\infty.
\end{equation*}
\end{proposition}
By linearizing the equation of \eqref{scalar nonlocal tw-parameter s} near $\mathcal{W}=1$ and changing $1-\mathcal{W}=\hat{\mathcal{W}}$, we have
$$J\ast\hat{\mathcal{W}}-\hat{\mathcal{W}}+c\hat{\mathcal{W}}'+f'(1;q)\hat{\mathcal{W}}=0.$$
Define
$I_2(\mu)=\int_{\mathbb{R}}\hat{\mathcal{W}}e^{\mu\xi}d\xi.$
Then, by multiplying $e^{\mu \xi}$ and integral on $\mathbb{R}$, we obtain
$$I_2(\mu)\Big(1-f'(1;q)+\mu c-\int_{\mathbb{R}}J(y)e^{\mu y}dy\Big)=0.$$
Notice that, $I_1(\mu)$ is a convex function.
Since $\int_{\mathbb{R}}J(y)e^{\mu y}dy=1$ when $\mu=0$, $\int_{\mathbb{R}}J(y)e^{\mu y}dy\to \infty$ as $\mu\to\infty$, and $f'(1;q)<0$, \eqref{nonlocal root} admits the unique positive root. Then, the proof of Proposition \ref{prop: asy tw - infty nonlocal} follows from the similar argument as Theorem 1.6 in \cite{Coville}.

\subsection{Construction of the super-solution}\label{subsec-3-1}

\noindent

Under the assumption (A1) and \eqref{assumption on J}, from Theorem 1.6 in \cite{Coville}, for each $q\geq0$, there exists a unique minimal traveling wave(up to a translation), and the minimal speed $c_{NL}^*(q)$ is continuous for all $q\ge 0$ by the assumption (A2).
Moreover, it follows from the assumption (A3) that $c_{NL}^*(q)$ is nondecreasing on $q$. Thus, we immediately obtain the following result by assumptions (A6),(A7), and Remark~\ref{rk:a6+a7}.

\begin{lemma}\label{lem: th1-part1 nonlocal}
Assume that assumptions (A1)-(A3), (A6), and (A7) hold. Then there exists a threshold $q^*\in[q_1,q_2)$ such that \eqref{def of threshold scalar nonlocal} holds.
\end{lemma}

Thanks to Lemma~\ref{lem: th1-part1 nonlocal}, to prove Theorem~\ref{th: threshold scalar equation nonlocal}, it suffices to show that
\eqref{asy tw threshold scalar nonlocal} holds if and only if $q=q^*$.
Let $\mathcal{W}_{q^*}$ be the minimal traveling wave of \eqref{scalar nonlocal tw-parameter s} with $q=q^*$
and 
speed $c_{NL}^*(q^*)=c_0^*$ defined as \eqref{formula of c_NL}. For simplicity, we denote $\mathcal{W}_*:=\mathcal{W}_{q^*}$.
Similar as the proof of Theorem \ref{th: threshold scalar equation}, the first and the most involved step is to show that if $q=q^*$, then \eqref{asy tw threshold scalar nonlocal} holds.
To do this, we shall use the contradiction argument again.
Assume that \eqref{asy tw threshold scalar nonlocal} is not true. Then, from \eqref{decay-U-linear}  
it holds that
\bea\label{AS-U-infty-for-contradiction scalar nonlocal}
\lim_{\xi\rightarrow+\infty}\frac{\mathcal{W}_{*}(\xi)}{\xi e^{-\lambda_0\xi}}=A_0\quad \mbox{for some $A_0>0$,}
\eea
where $\lambda_0$ is defined in Remark \ref{rk:a6+a7}.

Under the condition \eqref{AS-U-infty-for-contradiction scalar nonlocal},
we shall prove the following proposition.

\begin{proposition}\label{Prop-supersol nonlocal}
Assume that assumptions (A1)-(A3), (A6), and (A7) hold. In addition, if \eqref{AS-U-infty-for-contradiction scalar nonlocal} holds,
then there exists an auxiliary continuous function $\mathcal{R}_w(\xi)$ defined in $\mathbb{R}$ satisfying
\begin{equation}\label{decay rate of Rw nonlocal}
\mathcal{R}_w(\xi)=O(\xi e^{-\lambda_0\xi}) \quad \mbox{as $\xi\to\infty$},
\end{equation}
such that $\bar{\mathcal{W}}(\xi):=\min\{\mathcal{W}_{*}(\xi)-\mathcal{R}_w(\xi),1\}\geq (\not\equiv) 0$
satisfies
\bea\label{tw super solution scalar nonlocal}
\mathcal{N}_0[\bar{\mathcal{W}}]:=J\ast\bar{\mathcal{W}}-\bar{\mathcal{W}}+c^*_0\bar{\mathcal{W}}'+f(\bar{\mathcal{W}};q^*+\delta_0)\le0,\quad  \mbox{a.e. in $\mathbb{R}$},
\eea
for all sufficiently small $\delta_0>0$, where
${\bar{\mathcal{W}}}'(\xi_0^{\pm})$ exists and
${\bar{\mathcal{W}}}'(\xi_0^+)\leq {\bar{\mathcal{W}}}'(\xi_0^-)$
if $\bar{\mathcal{W}}'$ is not continuous at $\xi_0$.
\end{proposition}

In general, we may call such $\bar{\mathcal{W}}$ a {\em super-solution} of $\mathcal{N}_0[\cdot]=0$.
Next, we shall construct $\mathcal{R}_w(\xi)$ like what we have done in \S \ref{subsec-2-1} to prove Proposition~\ref{Prop-supersol nonlocal}.
Hereafter, assumptions (A1)-(A3), (A6), and (A7)  are always assumed.

\begin{figure}
\begin{center}
\begin{tikzpicture}[scale = 1.1]
\draw[thick](-6,0) -- (7,0) node[right] {$\xi$};
\draw [semithick] (-6, -0.1) to [ out=0, in=150] (-4,-0.4) to [out=30, in=230] (-3.5,0)  to [out= 40, in=180] (-3,0.2) to [out=15, in=220] (1,1.6) to [out=70,in=190] (1.5,2) to [out=0,in=170] (7,0.3);
\node[below] at (1,1.2) {$\xi_{1}+\delta_1$};
\node[below] at (0.8,0) {$0$};
\draw[dashed] [thick] (-3.5,0)-- (-3.5,0.3);
\node[above] at (-3.5,0.3) {$\xi_2$};
\draw[dashed] [thick] (-3,0.2)-- (-3,-0.5);
\draw[dashed] [thick] (-4,0)-- (-4,1);
\node[above] at (-4,0.8) {$\xi_2-\delta_3$};
\draw[dashed] [thick] (1,0)-- (1,0.65);
\draw[dashed] [thick] (1,1.15)-- (1,1.6);
\node[below] at (-3,-0.4) {$\xi_2+\delta_2$};
\draw [thin] (-3.85,-0.31) arc [radius=0.2, start angle=40, end angle= 140];
\node[above] at (-4,-0.36) {$\alpha_3$};
\draw [thin] (-2.8,0.255) arc [radius=0.2, start angle=30, end angle= 175];
\node[above] at (-3,0.3) {$\alpha_2$};
\draw [thin] (1.1,1.8) arc [radius=0.2, start angle=70, end angle= 220];
\node[above] at (1,1.8) {$\alpha_1$};
\end{tikzpicture}
\caption{the construction of $\mathcal{R}_w(\xi)$.}\label{Figure scalar nl}
\end{center}
\end{figure}
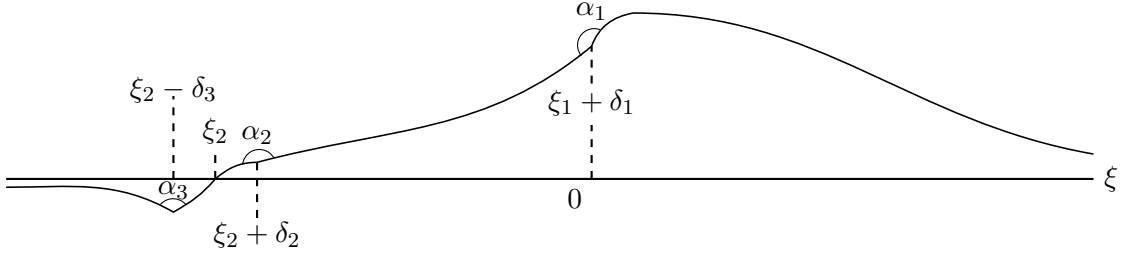

Let $\xi_1,\xi_2$ be chosen like that in Lemma \ref{lm: divide R to 3 parts}. Similar as $R_w(\xi)$ constructed in \S\ref{subsec-2-1}, we shall construct auxiliary function $\mathcal{R}_w(\xi)$ (also see Figure \ref{Figure scalar nl}) as follows:
\begin{equation}\label{definition of Rw nonlocal}
\mathcal{R}_w(\xi)=\begin{cases}
\varepsilon_1\sigma(\xi) e^{-\lambda_0\xi},&\ \ \mbox{for}\ \ \xi\ge\xi_{1}+\delta_1,\\
\varepsilon_2 e^{\lambda_1\xi},&\ \ \mbox{for}\ \ \xi_{2}+\delta_2\le\xi\le\xi_{1}+\delta_1,\\
\varepsilon_3\sin(\delta_4(\xi-\xi_2)),&\ \ \mbox{for}\ \ \xi_2-\delta_3\le\xi\le\xi_2+\delta_2,\\
-\varepsilon_4e^{\lambda_2\xi},&\ \ \mbox{for}\ \ \xi\le\xi_2-\delta_3,
\end{cases}
\end{equation}
where $\delta_{i=1,\cdots,4}>0$, $\lambda_{1,2}>0$, and $\sigma(\xi)>0$ will be
determined such that $\bar{\mathcal{W}}(\xi)$ satisfies \eqref{tw super solution scalar nonlocal}. Moreover, we should choose
$\varepsilon_{j=1,\cdots,4}\ll A_0$ ($A_0$ is defined in \eqref{AS-U-infty-for-contradiction scalar nonlocal}) such that $\mathcal{R}_w(\xi)\ll {\mathcal{W}}_*(\xi)$ and $\bar{\mathcal{W}}(\xi)$ is continuous for all $\xi\in\mathbb{R}$.

Since $f(\cdot;q^*)\in C^2$, there exist $K_1>0$ and $K_2>0$ such that
\bea\label{K1K2 nonlocal}
|f''({\mathcal{W}}_*(\xi);q^*)|< K_1,\quad |f'({\mathcal{W}}_*(\xi);q^*)|< K_2\quad \mbox{for all}\quad\xi\in\mathbb{R}.
\eea
We set $\lambda_1>0$ 
large enough such that
\begin{equation}\label{condition on lambda 1 nonlocal}
-c_0^*\lambda_1+1+K_2<0\ \text{and}\ \lambda_1>\frac{4K_2}{c_0^*}.
\end{equation}

Furthermore, there exists $K_3>0$ such that
\begin{equation}\label{condition on xi_2 scalar nonlocal}
f'({\mathcal{W}}_*(\xi); q^*)\le -K_3<0\quad\text{for all}\quad\xi\le \xi_2.
\end{equation}
Since the kernel $J$ has a compact support, without loss of generality, we assume $J\ge0$ on $[-L,L]$, and $J=0$ for $x\in(-\infty,-L)]\cup[L,\infty)$.
We define $\mu_0=\mu_{q^*,c^*_0}$ which is the unique positive root obtained from Proposition \ref{prop: asy tw - infty nonlocal} with $q=q^*$ and $c=c^*_0$.
Then, by setting $0<\lambda_2<\mu_0$ sufficiently small, we have
\bea\label{lambda 2 nonlocal}
1+K_3-e^{\lambda_2L}-c_0^*\lambda_2>0.
\eea

We now divide the proof into several steps.

\noindent{\bf{Step 1}:} We consider $\xi\in[\xi_1+\delta_1,\infty)$ where $\delta_1>0$ is small enough and will be determined in Step 2.
In this case, we have
\beaa
\mathcal{R}_w(\xi)=\varepsilon_1\sigma(\xi)\ e^{-\lambda_0\xi}
\eeaa
for some small $\varepsilon_1\ll A_0$.

Note that $\mathcal{W}_*$ satisfies \eqref{scalar nonlocal tw-parameter s} with $c=c_0^*$.
By some  straightforward computations, we have
\begin{equation}\label{N0 inequality step 1 nonlocal}
\begin{aligned}
\mathcal{N}_0[\bar{\mathcal{W}}]=&-J\ast \mathcal{R}_w+\mathcal{R}_w-c_0^*\mathcal{R}'_w-f(\mathcal{W}_*;q^*)+f(\mathcal{W}_*-\mathcal{R}_w;q^*+\delta_0)\\
=&-J\ast \mathcal{R}_w+\mathcal{R}_w-c_0^*\mathcal{R}'_w-f(\mathcal{W}_*;q^*)+f(\mathcal{W}_*-\mathcal{R}_w;q^*)\\
&-f(\mathcal{W}_*-\mathcal{R}_w;q^*)+f(\mathcal{W}_*-\mathcal{R}_w;q^*+\delta_0).
\end{aligned}
\end{equation}
By assumptions (A1) and (A2), and the statement (1) of Lemma \ref{lm: divide R to 3 parts}, since $\mathcal{R}_w\ll \mathcal{W}_*\ll 1$ for $\xi\in[\xi_1+\delta_1,\infty)$, we have
\begin{equation}\label{estimate 1 nonlocal}
-f(\mathcal{W}_*;q^*)+f(\mathcal{W}_*-\mathcal{R}_w;q^*)= -f'(0;q^*)\mathcal{R}_w+f''(0;q^*)(\frac{\mathcal{R}^2_w}{2}-\mathcal{W}_*\mathcal{R}_w)+o((\mathcal{W}_*)^2),
\end{equation}
\begin{equation}\label{estimate 2 nonlocal}
-f(\mathcal{W}_*-\mathcal{R}_w;q^*)+f(\mathcal{W}_*-\mathcal{R}_w;q^*+\delta_0)\le C_1\delta_0(\mathcal{W}_*-\mathcal{R}_w)^2+o((\mathcal{W}_*)^2).
\end{equation}

From \eqref{formula of c_NL}, \eqref{K1K2 nonlocal}, \eqref{N0 inequality step 1 nonlocal}, \eqref{estimate 1 nonlocal}, \eqref{estimate 2 nonlocal}, and Lemma \ref{lm: divide R to 3 parts}, we have
\begin{equation}\label{estimate 4 nonlocal}
\begin{aligned}
\mathcal{N}_0[\bar{\mathcal{W}}]\le &-\varepsilon_1e^{-\lambda_0\xi}\Big(\int_{\mathbb{R}}J(y)[\sigma(\xi-y)-\sigma(\xi)] e^{\lambda_0y}dy\Big)-c^*_0\sigma'e^{-\lambda_0\xi}\\
&+K_1(\frac{\mathcal{R}_w^2}{2}+\mathcal{W}_*\mathcal{R}_w)+C_1\delta_0\mathcal{W}_*^2+o((\mathcal{W}_*)^2).
\end{aligned}
\end{equation}
Let $h(\lambda)$ be defined as that in Remark \ref{rm:lambda_0}. Since $(h(\lambda)/\lambda)'=0$ when $\lambda=\lambda_0$, from \eqref{formula of c_NL}, we get
\begin{equation}\label{c*0 formula}
c^*_0=\int_{\mathbb{R}}yJ(y)e^{\lambda_0y}dy.
\end{equation}
Then, it follows from \eqref{estimate 4 nonlocal} and \eqref{c*0 formula} that
\begin{equation}\label{estimate 3 nonlocal}
\begin{aligned}
\mathcal{N}_0[\bar{\mathcal{W}}]\le &-\varepsilon_1e^{-\lambda_0\xi}\int_{\mathbb{R}}J(y)[\sigma(\xi-y)-\sigma(\xi)+y\sigma'(\xi)] e^{\lambda_0y}dy\\
&+K_1(\frac{\mathcal{R}_w^2}{2}+\mathcal{W}_*\mathcal{R}_w)+C_1\delta_0\mathcal{W}_*^2+o((\mathcal{W}_*)^2).
\end{aligned}
\end{equation}

Now, we define
$$\sigma(\xi):=\frac{1}{\lambda_0^2}e^{-\frac{\lambda_0}{2}(\xi-\xi_1)}-\frac{1}{\lambda_0^2}+\frac{1}{\lambda_0}\xi-\frac{1}{\lambda_0}\xi_1$$
which satisfies 
$$\sigma(\xi_1)=0,\ \sigma'(\xi)=\frac{1}{\lambda_0}-\frac{1}{2\lambda_0}e^{-\frac{\lambda_0}{2}(\xi-\xi_1)}.$$
 Moreover, $\sigma(\xi)= O(\xi)$ as $\xi\to\infty$ implies that $\mathcal{R}_w$ satisfies \eqref{decay rate of Rw nonlocal}.

By some straightforward computation, we have
\begin{equation*}
\int_{\mathbb{R}}J(y)[\sigma(\xi-y)-\sigma(\xi)+y\sigma'(\xi)] e^{\lambda_0y}dy=\frac{1}{\lambda_0^2}e^{-\frac{\lambda_0}{2}(\xi-\xi_1)}\int_{\mathbb{R}}J(y)e^{\lambda_0 y}[e^{\frac{\lambda_0y}{2}}-1-\frac{\lambda_0y}{2}]dy.
\end{equation*}
Notice that, the function 
$$g(y):=e^{\frac{\lambda_0y}{2}}-1-\frac{\lambda_0y}{2}\ge 0$$
 is convex and obtains minimum at $y=0$, and $J(x)=0$ for $|x|>L$. Therefore, we assert that there exists $K_4>0$ independent on $\xi_1$ such that
\begin{equation}\label{estimate K4 nonlocal}
-\varepsilon_1e^{-\lambda_0\xi}\int_{\mathbb{R}}J(y)[\sigma(\xi-y)-\sigma(\xi)+y\sigma'(\xi)] e^{\lambda_0y}dy\le -\varepsilon_1K_4e^{-\lambda_0\xi}e^{\frac{-\lambda_0(\xi-\xi_1)}{2}}.
\end{equation}

Then, from \eqref{estimate 3 nonlocal} and \eqref{estimate K4 nonlocal}, up to enlarging $\xi_1$ if necessary, we always have 
\beaa
\mathcal{N}_0[\bar{\mathcal{W}}]\le -\varepsilon_1K_4e^{-\frac{\lambda_0}{2}(\xi-\xi_1)}e^{-\lambda_0\xi}+K_1(\frac{\mathcal{R}_w^2}{2}+\mathcal{W}_*\mathcal{R}_w)+C_1\delta_0\mathcal{W}_*^2+o((\mathcal{W}_*)^2)\le 0
\eeaa
for all sufficiently small $\delta_0>0$
since $\mathcal{R}_w^2(\xi)$, $\mathcal{W}_*\mathcal{R}_w(\xi)$, and  $\mathcal{W}_*^2(\xi)$ are $o(e^{-\frac{3\lambda_0}{2}\xi})$ for $\xi\ge \xi_1$ from \eqref{AS-U-infty-for-contradiction scalar nonlocal} and the definition of $\mathcal{R}_w$.

\medskip

\noindent{\bf{Step 2:}} We consider $\xi\in[\xi_2+\delta_2,\xi_1+\delta_1]$ for some small $\delta_2>0$, and sufficiently small $\delta_1>0$ satisfying
\bea\label{cond delta 1 scalar nonlocal}
\frac{3}{2}e^{-\frac{\lambda_0\delta_1}{2}}+\delta_1\lambda_0<2.
\eea
From the definition of $\mathcal{R}_w$ in Step 1, it is easy to check that $\mathcal{R}'_w((\xi_1+\delta_1)^+)>0$ under the condition \eqref{cond delta 1 scalar nonlocal}.
In this case, we have $\mathcal{R}_w(\xi)=\varepsilon_2 e^{\lambda_1\xi}$
for some large $\lambda_1>0$ satisfying \eqref{condition on lambda 1 nonlocal}.

We first choose
\begin{equation}\label{condition on epsilon 2}
\varepsilon_2=\frac{\varepsilon_1}{\lambda^2_0}\Big(e^{-\frac{\lambda_0\delta_1}{2}}-1+\delta_1\lambda_0\Big)e^{-(\lambda_0+\lambda_1)(\xi_1+\delta_1)}>0
\end{equation}
such that $\mathcal{R}_w(\xi)$ is continuous at $\xi=\xi_{1}+\delta_1$. Then, from \eqref{condition on epsilon 2}, we have
$$\mathcal{R}'_w((\xi_1+\delta_1)^{+})=\varepsilon_1\sigma'(\xi_1+\delta_1)e^{-\lambda_0(\xi_1+\delta_1)}- \lambda_0\mathcal{R}_w(\xi_1+\delta_1),$$
$$\mathcal{R}'_w((\xi_1+\delta_1)^{-})=\lambda_1\mathcal{R}_w(\xi_1+\delta_1),$$
and $\mathcal{R}'_w((\xi_1+\delta_1)^{+})>\mathcal{R}'_w((\xi_1+\delta_1)^{-})$ is equivalent to
$$2(\lambda_0+\lambda_1)(e^{-\frac{\lambda_0\delta_1}{2}}-1+\delta_1\lambda_0)+\lambda_0e^{-\frac{\lambda_0\delta_1}{2}}<2\lambda_0,$$
which holds by taking $\delta_1$ sufficiently small.
This implies that $\angle\alpha_1<180^{\circ}$.

Since $J\ast \mathcal{R}_w\ge 0$, by some  straightforward computations, we have
\begin{equation*}
\begin{aligned}
\mathcal{N}_0[\bar{\mathcal{W}}]
\le&-(c_0^*\lambda_1-1)\mathcal{R}_w-f(\mathcal{W}_*;q^*)+f(\mathcal{W}_*-\mathcal{R}_w;q^*)\\
&-f(\mathcal{W}_*-\mathcal{R}_w;q^*)+f(\mathcal{W}_*-\mathcal{R}_w;q^*+\delta_0).
\end{aligned}
\end{equation*}
Thanks to \eqref{K1K2 nonlocal}, we have
$$-f(\mathcal{W}_*;q^*)+f(\mathcal{W}_*-\mathcal{R}_w;q^*)< K_2\mathcal{R}_w.$$
Moreover, by the assumption (A2), 
$$-f(\mathcal{W}_*-\mathcal{R}_w;q^*)+f(\mathcal{W}_*-\mathcal{R}_w;q^*+\delta_0)\le L_0\delta_0.$$
Then, 
since
$\lambda_1$ satisfies \eqref{condition on lambda 1 nonlocal}, we have
$$L_0\delta_0<\varepsilon_2(c_0^*\lambda_1-1-K_2) e^{\lambda_1(\xi_2+\delta_2)}$$
for all sufficiently small $\delta_0>0$,
which implies that $\mathcal{N}_0[\bar{\mathcal W}]\le 0$ for all $\xi\in[\xi_2+\delta_2,\xi_1+\delta_1]$.

\medskip

\noindent{\bf{Step 3:}} We consider $\xi\in[\xi_2-\delta_3,\xi_2+\delta_2]$ for some small $\delta_2,\delta_3>0$. In this case, we have $\mathcal{R}_w=\varepsilon_3\sin(\delta_4(\xi-\xi_2))$.
By applying the same argument as Claim \ref{cl scalar} we can obtain a claim as follows.
\begin{claim}
For any $\delta_2$ with $\delta_2>\frac{1}{\lambda_1}$, there exist
$\varepsilon_3>0$ and small $\delta_4>0$ satisfying
\bea\label{epsilon 3 nonlocal}
\varepsilon_3= \frac{\varepsilon_2e^{\lambda_1(\xi_2+\delta_2)}}{\sin(\delta_4\delta_2)}>0
\eea
such that $\mathcal{R}_w((\xi_2+\delta_2)^+)=\mathcal{R}_w((\xi_2+\delta_2)^-)$ and $\angle\alpha_2<180^{\circ}$.
\end{claim}

Next, we verify the differential inequality of $\mathcal{N}_0[\bar{\mathcal W}]$ for $\xi\in[\xi_2-\delta_3,\xi_2+\delta_2]$.
Since the kernel $J$ has a compact support,
by some straightforward computations, we have
\begin{equation*}
\begin{aligned}
\mathcal{N}_0[\bar{\mathcal W}]=&\varepsilon_3\int_{-L}^LJ(y)\Big(\sin(\delta_4(\xi-\xi_2)-\sin(\delta_4(\xi-y-\xi_2))\Big)dy-c_0^*\varepsilon_3\delta_4\cos(\delta_4(\xi-\xi_2))\\
&-f(\mathcal{W}_*;s^*)+f(\mathcal{W}_*-\mathcal{R}_w;s^*)-f(\mathcal{W}_*-\mathcal{R}_w;s^*)+f(\mathcal{W}_*-\mathcal{R}_w;s^*+\delta_0)\\
\le &\varepsilon_3\int_{-L}^L\Big|\sin(\delta_4(\xi-\xi_2))-\sin(\delta_4(\xi-y-\xi_2))\Big|dy\\
&+K_2\varepsilon_3\sin(\delta_4(\xi-\xi_2))-c_0^*\varepsilon_3\delta_4\cos(\delta_4(\xi-\xi_2))+L_0\delta_0.
\end{aligned}
\end{equation*}

We first focus on $\xi\in[\xi_2,\xi_2+\delta_2]$. Notice that, the integral is defined on a bounded domain and we always set $\delta_4$ small. Up to decreasing $\delta_4$ if necessary, by Taylor series, we have
$$\sin(\delta_4(\xi-\xi_2-y))-\sin(\delta_4(\xi-\xi_2))\sim-y\delta_4^2\cos(\delta_4(\xi-\xi_2))-\frac{y^2\delta_4^4}{2}\sin(\delta_4(\xi-\xi_2)).$$
Then, by setting $\delta_4<c_0^*/2L$,
\begin{equation}\label{inequilty 1}
|y\delta_4^2\cos(\delta_4(\xi-\xi_2))|<c_0^*\delta_4\cos(\delta_4(\xi-\xi_2))/2.
\end{equation}
Therefore, we obtain from \eqref{inequilty 1} that
\begin{equation}\label{inequality 2}
\mathcal{N}_0[\bar{\mathcal W}]\le -\varepsilon_3\frac{c_0^*\delta_4}{2}\cos(\delta_4(\xi-\xi_2))+\varepsilon_3(K_2+\frac{L^2\delta_4^4}{2})\sin(\delta_4(\xi-\xi_2))+L_0\delta_0.
\end{equation}

By \eqref{epsilon 3 nonlocal} and the fact $x\cos x\to\sin x$ as $x\to 0$,
\beaa
\min_{\xi\in[\xi_2,\xi_2+\delta_2]}\frac{\delta_4\varepsilon_3c_0^*}{2}\cos(\delta_4(\xi-\xi_2))\to\frac{c_0^*\varepsilon_2e^{\lambda_1(\xi_2+\delta_2)}}{2\delta_2}
=\frac{c^*_0\mathcal{R}_w(\xi_2+\delta_2)}{2\delta_2}
\quad\text{as}\ \delta_4\to0.
\eeaa
Then, by \eqref{condition on lambda 1 nonlocal}, we can choose
$\delta_2\in(1/\lambda_1,c_0^*/4K_2)$ such that
\beaa
\frac{c^*_0\mathcal{R}_w(\xi_2+\delta_2)}{2\delta_2}>2 K_2\mathcal{R}_w(\xi_2+\delta_2)
\quad\text{for small}\ \delta_4<(\frac{2K_2}{\varepsilon_3L^2})^{\frac{1}{4}}.
\eeaa
Thus, we have
\beaa
\min_{\xi\in[\xi_2,\xi_2+\delta_2]}\frac{\delta_4\varepsilon_3c_0^*}{2}\cos(\delta_4(\xi-\xi_2))>(K_2+\frac{L^2\delta_4^4}{2})\mathcal{R}_w(\xi),
\eeaa
for all sufficiently small $\delta_4>0$.
Then, from \eqref{inequality 2}, up to decreasing $\delta_0>0$ if necessary, 
we see that $\mathcal{N}_0[\bar{\mathcal W}]\le 0$ for $\xi\in[\xi_2,\xi_2+\delta_2]$.

For $\xi\in[\xi_2-\delta_3,\xi_2]$, by the same argument
we can set $\delta_3>0$ small enough such that
$\mathcal{N}_0[\bar{\mathcal W}]\le 0$. This completes the Step 3.

\medskip

\noindent{\bf{Step 4:}} We consider $\xi\in(-\infty,\xi_2-\delta_3]$. In this case, we have $\mathcal{R}_w(\xi)=-\varepsilon_4e^{\lambda_2\xi}<0$.
Recall that we choose $0<\lambda_2<\mu_0$
and 
$$1-\mathcal{W}_*(\xi)\sim C_2 e^{\tilde\lambda\xi}\ \ \text{as}\ \ \xi\to-\infty.$$
Then, there exists $M_1>0$ such that 
$$\bar{\mathcal W}=\min\{\mathcal{W}_*-\mathcal{R}_w,1\}\equiv 1\ \ \text{for all}\ \ \xi\le -M_1,$$
and thus 
$$\mathcal{N}_0[\bar{\mathcal W}]\le 0\  \ \text{for all}\ \ \xi\le -M_1.$$ Therefore, we only need to show 
$$\mathcal{N}_0[\bar{\mathcal W}]\le 0\ \  \text{for all}\ \ -M_1\le\xi\le -\xi_2-\delta_3.$$

We first choose 
$$\varepsilon_4=\varepsilon_3\sin(\delta_4\delta_3)/e^{\lambda_2(\xi_2-\delta_3)}$$
such that $\mathcal{R}_w$ is continuous at $\xi_2-\delta_3$. It is easy to check that 
$$\mathcal{R}'_w((\xi_2-\delta_3)^+)>0>\mathcal{R}'_w((\xi_2-\delta_3)^-),$$
 and thus $\angle\alpha_3< 180^{\circ}$.

Since the kernel $J$ is trivial outside of $[-L,L]$, by some straightforward computations, we have
\begin{equation*}
\begin{aligned}
\mathcal{N}_0[\bar{\mathcal W}]
\le&-(e^{\lambda_2L}+c_0^*\lambda_2-1)\mathcal{R}_w-f(\mathcal{W}_*;q^*)+f(\mathcal{W}_*-\mathcal{R}_w;q^*)\\
&-f(\mathcal{W}_*-\mathcal{R}_w;q^*)+f(\mathcal{W}_*-\mathcal{R}_w;q^*+\delta_0).
\end{aligned}
\end{equation*}
From \eqref{condition on xi_2 scalar nonlocal} and $\mathcal{R}_w\le 0$, we have 
$$-f(\mathcal{W}_*;q^*)+f(\mathcal{W}_*-\mathcal{R}_w;q^*)< K_3\mathcal{R}_w<0.$$
Together with the assumption (A2), we have
\beaa
\mathcal{N}_0[\bar{\mathcal W}]\leq -(e^{\lambda_2L}+c_0^*\lambda_2-1-K_3)\mathcal{R}_w+L_0\delta_0\quad\text{for all}\quad \xi\in[-M,\xi_2-\delta_3].
\eeaa
In view of \eqref{lambda 2 nonlocal},
we can assert that 
$$\mathcal{N}_0[\bar{\mathcal W}]\le 0\ \ \text{for all}\ \ \xi\in[-M,\xi_2-\delta_3],$$
provided that $\delta_0$ is sufficiently small.
This completes the Step 4.

\subsection{Proof of Theorem \ref{th: threshold scalar equation nonlocal}}

\noindent

We are ready to prove Theorem~\ref{th: threshold scalar equation nonlocal} as follows.

\begin{proof}[Proof of Theorem~\ref{th: threshold scalar equation nonlocal}]
In view of Lemma~\ref{lem: th1-part1 nonlocal}, we have obtained \eqref{def of threshold scalar nonlocal}.
It suffices to show that \eqref{asy tw threshold scalar nonlocal} holds if and only if $q=q^*$.
From the discussion from Step 1 to Step 4 in \S \ref{subsec-3-1},
we are now equipped with an auxiliary function $\mathcal{R}_w(\xi)$ 
defined as in \eqref{definition of Rw nonlocal} such that 
$$\bar{\mathcal W} (\xi)=\min \{\mathcal{W}_*(\xi)-\mathcal{R}_w(\xi),1\},$$
 which is independent of the choice of all sufficiently small $\delta_0>0$, forms
a super-solution satisfying \eqref{tw super solution scalar nonlocal}. By the comparison argument used in the proof of Theorem \ref{th: threshold scalar equation}, similarly we can show
\beaa
q=q^* \quad \Longrightarrow \quad \mbox{\eqref{asy tw threshold scalar nonlocal} holds}.
\eeaa
Therefore, it suffices to prove
\bea\label{th1:goal-2 nonlocal}
\mbox{\eqref{asy tw threshold scalar nonlocal} holds} \quad \Longrightarrow \quad q=q^*
\eea
by the sliding method.

We assume by contradiction that there exists $q_0\in(0,q^*)$ such that
the corresponding minimal traveling wave satisfies
\bea\label{W-S0+infty nonlocal}
\mathcal{W}_{q_0}(\xi)=B_0 e^{-\lambda_0\xi}+o(e^{-\lambda_0\xi})\quad\text{as}\quad \xi\to+\infty
\eea
for some $B_0>0$. For $\xi\approx -\infty$, from Proposition \ref{prop: asy tw - infty nonlocal}, we have
\bea\label{W-S0-infty nonlocal}
1-\mathcal{W}_{q_0}(\xi)=C_0 e^{\tilde\mu_0\xi}+o(e^{\tilde\mu_0\xi})\quad\text{as}\quad \xi\to-\infty
\eea
for some $C_0>0$, where $\tilde\mu_0=\mu_{s_0,c_0^*}$.
Recall that the asymptotic behavior of $\mathcal{W}_{q^*}$ at $\pm\infty$ satisfies
\bea\label{W-S*-pm-infty nonlocal}
\mathcal{W}_{q^*}(\xi)=B e^{-\lambda_0\xi}+o(e^{-\lambda_0\xi})\ \text{as}\ \xi\to+\infty;\ \ 1-\mathcal{W}_{q^*}(\xi)=C e^{\mu_0\xi}+o(e^{\mu_0\xi})\ \text{as}\ \xi\to-\infty
\eea
for some $B,C>0$, where $\mu_0=\mu_{q^*,c^*_0}$. In view of the assumption (A3), we have $\mu_0>\tilde\mu_0$ since $q^*>q_0$.
Combining \eqref{W-S0+infty nonlocal}, \eqref{W-S0-infty nonlocal} and \eqref{W-S*-pm-infty nonlocal}, there exists $0<L<\infty$ sufficiently large such that
$\mathcal{W}_{q^*}(\xi-L)> \mathcal{W}_{q_0}(\xi)$ for all $\xi\in\mathbb{R}$. Now, we define
\beaa
L^*:=\inf\{L\in\mathbb{R}\ |\ \mathcal{W}_{q^*}(\xi-L)\ge \mathcal{W}_{q_0}(\xi)\ \text{for all}\ \xi\in\mathbb{R}\}.
\eeaa
By the continuity, we have 
$$\mathcal{W}_{q^*}(\xi-L^*)\geq \mathcal{W}_{q_0}(\xi)\ \text{for all}\ \ \xi\in\mathbb{R}.$$
 If there exists $\xi^*\in\mathbb{R}$ such that
$\mathcal{W}_{q^*}(\xi^*-L^*)= \mathcal{W}_{q_0}(\xi^*)$, by the strong maximum principle, we have 
$$\mathcal{W}_{q^*}(\xi-L^*)=\mathcal{W}_{q_0}(\xi)\ \ \text{for all}\ \ \xi\in\mathbb{R},$$
which is impossible since $\mathcal{W}_{q^*}(\cdot-L^*)$ and $\mathcal{W}_{q_0}(\cdot)$ satisfy different equations. Consequently,
$$\mathcal{W}_{q^*}(\xi-L^*)> \mathcal{W}_{q_0}(\xi)\ \ \text{for all}\ \ \xi\in\mathbb{R}.$$ 
In particular, we have
\beaa
\lim_{\xi\to+\infty}\frac{\mathcal{W}_{q^*}(\xi-L^*)}{\mathcal{W}_{q_0}(\xi)}\geq1.
\eeaa
Furthermore, we can claim that
\bea\label{limit=1 nonlocal}
\lim_{\xi\to+\infty}\frac{\mathcal{W}_{q^*}(\xi-L^*)}{\mathcal{W}_{q_0}(\xi)}=1.
\eea
Otherwise, if the limit in \eqref{limit=1 nonlocal} is strictly bigger than 1, together with $\mu_0>\tilde \mu_0$ and
\beaa
\lim_{\xi\to-\infty}\frac{1-\mathcal{W}_{q^*}(\xi-L^*)}{1-\mathcal{W}_{q_0}(\xi)}=0,
\eeaa
we can easily  find $\varepsilon>0$ sufficiently small such that
$$\mathcal{W}_{q^*}(\xi-(L^*+\varepsilon))> \mathcal{W}_{q_0}(\xi)\ \ \text{for all}\ \ \xi\in\mathbb{R},$$
which contradicts the definition of $L^*$.
As a result, from \eqref{W-S0+infty nonlocal}, \eqref{W-S*-pm-infty nonlocal} and \eqref{limit=1 nonlocal}, we obtain $B_0=Be^{L^*}$.

On the other hand, we set $\widehat{\mathcal{W}}(\xi)=\mathcal{W}_{q^*}(\xi-L^*)-\mathcal{W}_{s_0}(\xi)$. Then $\widehat{\mathcal{W}}(\xi)$ satisfies
\bea\label{W-hat-eq2 nonlocal}
J\ast\widehat{\mathcal{W}}+c^*_0\widehat{\mathcal{W}}'+(f'(0)-1)\widehat{\mathcal{W}}+J(\xi)=0, \quad \xi\in\mathbb{R},
\eea
where
$$J(\xi)=f(\mathcal{W}_{s^*};s^*)- f'(0)\mathcal{W}_{s^*}-f(\mathcal{W}_{s_0};s_0)+ f'(0)\mathcal{W}_{s_0}.$$
By the assumption  (A1) and Taylor's Theorem, there exist $\eta_1\in(0, W_{s^*})$ and $\eta_2\in(0,W_{s_0})$ such that
\beaa
J(\xi)=J_1(\xi)+J_2(\xi)
\eeaa
where
$$J_1(\xi):=f''(\eta_1;q^*)(\mathcal{W}_{q^*}+\mathcal{W}_{q_0})\widehat{\mathcal{W}},$$ 
$$J_2(\xi):=[f''(\eta_1;q^*)-f''(\eta_2;q_0)]\mathcal{W}^2_{q_0}.$$
 It is easy to see that $J_1(\xi)=o(\widehat{\mathcal{W}})$ for $\xi\approx+\infty$. Next, we will show $J_2(\xi)=o(\widehat{\mathcal{W}})$ for $\xi\approx+\infty$.

Since $f''(0;s^*)>f''(0;s_0)$ (from the assumption (A3)), we can find small $\delta>0$ such that
$$\min_{\eta\in[0,\delta]}f''(\eta;q^*)>\max_{\eta\in[0,\delta]}f''(\eta;q_0)$$
and thus
there exist $\kappa_1,\kappa_2>0$ such that
\bea\label{J-lower bound nonlocal}
\kappa_1e^{-2\lambda_0\xi}\ge J_2(\xi)=[f''(\eta_1;q^*)-f''(\eta_2;q_0)]\mathcal{W}^2_{q_0}(\xi)\ge \kappa_2 e^{-2\lambda_0\xi}\quad \mbox{for all large $\xi$}.
\eea

We now claim that $J_2(\xi)=o(\widehat{\mathcal{W}})$ as $\xi\to+\infty$.
For contradiction, we assume that it is not true. Then there exists $\{\xi_n\}$ with
$\xi_n\to+\infty$ as $n\to\infty$ such that for some $\kappa_3>0$,
\bea\label{kappa3 nonlocal}
\frac{J_2(\xi_n)}{\widehat{\mathcal{W}}(\xi_n)}\geq \kappa_3\quad \mbox{for all $n\in\mathbb{N}$.}
\eea
Set $\widehat{\mathcal{W}}(\xi)=\alpha(\xi)e^{-2\lambda_0\xi}$, where $\alpha(\xi)>0$ for all $\xi$.
By substituting it into \eqref{W-hat-eq2 nonlocal},
we have 
\begin{equation}\label{alpha-eq nonlocal}
\begin{aligned}
L(\xi):=&\Big(\int_{\mathbb{R}}J(y)\alpha(\xi-y)e^{2\lambda_0y}dy+(f'(0)-1-2\lambda_0c^*_0)\alpha(\xi)+ c^*_0\alpha'(\xi)\Big)e^{-2\lambda_0\xi}\\
&+J_1(\xi)+J_2(\xi)=0
\end{aligned}
\end{equation}
for all large $\xi$.
By \eqref{J-lower bound nonlocal} and \eqref{kappa3 nonlocal}, we have 
\bea\label{alpha-bdd nonlocal}
0<\alpha(\xi_n)\leq \frac{\kappa_1}{\kappa_3}\quad  \mbox{for all $n\in\mathbb{N}$.}
\eea
Now, we will reach a contradiction by dividing the behavior of $\alpha(\cdot)$ into two cases:
\begin{itemize}
    \item[(i)] $\alpha(\xi)$ oscillates for all large $\xi$;
    \item[(ii)] $\alpha(\xi)$ is monotone for all large $\xi$.
\end{itemize}

For case (i), there exist local minimum points $\eta_n$ of $\alpha$ with $\eta_n\to\infty$ as $n\to\infty$ such that
\beaa
\alpha(\eta_n)>0\quad\text{and}\quad \alpha'(\eta_n)=0\quad  \mbox{for all $n\in\mathbb{N}$.}
\eeaa
Without loss of generality, we also assume that 
\bea\label{local mini L}
\alpha(\eta_n)\ge\alpha(\xi)\quad\text{for all}\quad \xi\in[\eta_n-L,\eta_n+L].
\eea
Then from \eqref{formula of c_NL},  \eqref{alpha-eq nonlocal} yields that
\begin{equation*}
L(\eta_n)>\Big(\int_{\mathbb{R}}J(y)(\alpha(\eta_n-y)-\alpha(\eta_n))e^{2\lambda_0y}dy\Big)e^{-2\lambda_0\eta_n}+J_1(\xi_n)+J_2(\eta_n)
\end{equation*}
Together with \eqref{J-lower bound nonlocal} and $J_1(\xi)=o(\widehat{\mathcal{W}}(\xi))$,
from \eqref{alpha-eq nonlocal} and \eqref{local mini L}, we see that
\beaa
0=L(\eta_n)\geq o(1)\alpha(\eta_n)e^{-2\lambda_0\eta_n}+\kappa_2e^{-2\lambda_0\eta_n}>0
\eeaa
for all large $n$, which reaches a contradiction. 

For case (ii),
due to \eqref{alpha-bdd nonlocal}, there exists $\alpha_0\in[0, \kappa_1/\kappa_3]$
such that $\alpha(\xi)\to \alpha_0$ as $\xi\to\infty$. Hence, we can find subsequence $\{\eta_j\}$ that tends to $\infty$ such that $\alpha'(\eta_j)\to0$ and
$\alpha(\eta_j)\to \alpha_0$ as $n\to\infty$.
From \eqref{alpha-eq nonlocal} we deduce that
\beaa
0=L(\eta_j)\geq (o(1)+ \kappa_2)e^{-2\lambda_0\eta_j}>0
\eeaa
for all large $j$, which reaches a contradiction. 
Therefore, we have proved that
$J_2(\xi)=o(\widehat{\mathcal{W}})$ as $\xi\to\infty$.
Consequently, we have
\beaa
J(\xi)=J_1(\xi)+J_2(\xi)=o(\widehat{\mathcal{W}}(\xi))\quad \mbox{as $\xi\to\infty$.}
\eeaa

Now, by the proof of Proposition \ref{prop:correction-U-linear-decay}, we can assert that the asymptotic behavior of $\widehat{\mathcal{W}}(\xi)$ at $\xi=+\infty$ satisfies
\beaa
\widehat{\mathcal{W}}(\xi)=(C_1\xi+C_2)e^{-\beta \xi}+o(e^{-\beta\xi})\quad \mbox{as $\xi\to\infty$},
\eeaa
in which $C_1$ and $C_2$ can not be equal to $0$ simultaneously.
However, by $B_0=Be^{L^*}$, the asymptotic behaviors \eqref{W-S0+infty nonlocal} and \eqref{W-S*-pm-infty nonlocal} yield  $C_1=0$ and $C_2=0$, which reaches a contradiction.
Therefore, \eqref{th1:goal-2 nonlocal} holds, and the proof is complete.
\end{proof}

\bigskip

\noindent{\bf Acknowledgement.}
Maolin Zhou is supported by the National Key Research and Development Program of China (2021YFA1002400).
Chang-Hong Wu is supported by the Ministry of Science and Technology of Taiwan.
Dongyuan Xiao is supported by the Japan Society for the Promotion of Science P-23314.

\end{document}